\documentclass[12pt,reqno]{article}

\usepackage[usenames]{color}
\usepackage{amssymb}
\usepackage{amsmath}
\usepackage{amsthm}
\usepackage{amsfonts}
\usepackage{amscd}
\usepackage{graphicx}
\usepackage{diagbox}

\usepackage[colorlinks=true,
linkcolor=webgreen,
filecolor=webbrown,
citecolor=webgreen]{hyperref}

\definecolor{webgreen}{rgb}{0,.5,0}
\definecolor{webbrown}{rgb}{.6,0,0}

\usepackage{color}
\usepackage{fullpage}
\usepackage{float}

\usepackage{graphics}
\usepackage{latexsym}
\usepackage{epsf}
\usepackage{breakurl}
\DeclareMathOperator{\per}{per}
\def\suchthat{ \ : \ }

\newcommand{\seqnum}[1]{\href{https://oeis.org/#1}{\rm \underline{#1}}}
\def\Enn{\mathbb{N}}
\def\Zee{\mathbb{Z}}

\begin{document}

\theoremstyle{plain}
\newtheorem{theorem}{Theorem}
\newtheorem{corollary}[theorem]{Corollary}
\newtheorem{lemma}[theorem]{Lemma}
\newtheorem{proposition}[theorem]{Proposition}

\theoremstyle{definition}
\newtheorem{definition}[theorem]{Definition}
\newtheorem{example}[theorem]{Example}
\newtheorem{conjecture}[theorem]{Conjecture}

\theoremstyle{remark}
\newtheorem{remark}[theorem]{Remark}

\title{The Narayana Morphism and Related Words}

\author{Jeffrey Shallit\\
School of Computer Science \\
University of Waterloo\\
Waterloo, ON  N2L 3G1 \\
Canada \\
\href{mailto:shallit@uwaterloo.ca}{\tt shallit@uwaterloo.ca}}

\maketitle

\begin{abstract}
The Narayana morphism $\nu$ maps $0 \rightarrow 01$, $1 \rightarrow 2$,
$2 \rightarrow 0$ and has a fixed point 
${\bf n} = n_0 n_1 n_2 \cdots = {\tt 0120010120120}\cdots$.  In this
paper we study the properties of this word and related words using
automata theory.
\end{abstract}

\section{Introduction}

Consider the morphism $\nu: 0 \rightarrow 01$, $1 \rightarrow 2$,
$2 \rightarrow 0$,
which we call the {\it Narayana morphism} for reasons that will become
clear in a moment.  Its fixed point is
$${\bf n} = n_0 n_1 n_2 \cdots = {\tt 012001012012001200101200101201200101201200120010120120012001} \cdots,$$
the infinite {\it Narayana word}.  Up to renaming the letters, the
word $\bf n$ is
sequence \seqnum{A105083} in the
{\it On-Line Encyclopedia of Integer
Sequences} (OEIS) \cite{Sloane:2025}.

We can view $\nu$ as a higher-dimensional
generalization of the well-studied Fibonacci morphism 
\cite{Karhumaki:1983,Berstel:1986b,Fujita&Machida:1986}
$0 \rightarrow 01$, $1 \rightarrow 0$, and
$\bf n$ as a higher-dimensional
generalization of the Fibonacci word ${\bf f} = 
{\tt 01001010010010100101001} \cdots$ 
\cite{Recht&Rosenman:1947,Knuth&Morris&Pratt:1977}.  (See
Section~\ref{finalw} for more about this.)
Nevertheless, the morphism $\nu$ and the word $\bf n$
have received very little study.  The morphism $\nu$ was mentioned
briefly by Sirvent and Wang \cite{Sirvent&Wang:2002}, and
$\bf n$ appears (implicitly) in
\cite{Frougny&Masakova&Pelantova:2004,Ambroz&Masakova&Pelantova&Frougny:2006}
as an infinite word associated with a Parry number (with finite
R\'enyi expansion) that is a real root of the equation
$X^j = \sum_{0 \leq i < j} t_i X^i$.  In their notation, it corresponds to
$t_1 t_2 t_3 = 101$.  

In this paper we study properties of the word $\bf n$, and related
words, using
{\tt Walnut}, a free software tool for proving first-order properties
of sequences that can be generated with automata.  Although some of the
results proved here were already known (and indeed, some in much greater
generality), the virtue of the method we present here is its extreme
simplicity.  Often one only has to state the desired theorem and
{\tt Walnut} can prove it in a few seconds.  

The techniques involving automata are mostly familiar now,
and detailed explanations of the technique can be found in the author's book
\cite{Shallit:2023}.  However, there are a few novelties.  Along the way
we prove two conjectures, one due to Kimberling and Moses \cite{Kimberling&Moses:2010} and one due to Cloitre.

\section{The Narayana numbers and the Narayana numeration system}

It is natural to study the integer sequence
$(N_i)_{i \geq 0}$ defined by $N_i = | \nu^i (0) |$.  Table~\ref{nar}
gives the first few values of this sequence.
If we also define $N_{-1} = N_{-2} = 1$, then
a simple induction shows $N_i = N_{i-1} + N_{i-3}$ for $i \geq 1$,
and further that $|\nu^i (1)| = N_{i-2}$ and
$|\nu^i (2)| = N_{i-1}$
for $i \geq 0$.
\begin{table}[htb]
\begin{center}
\begin{tabular}{c|cccccccccccccccccccccccccccc}
$i$& 0&   1&   2&   3&   4&   5&   6&   7&   8&   9&  10&  11&  12&  13&  14&  15&  16  \\
\hline
$N_i$ &   1&   2&   3&   4&   6&   9&  13&  19&  28&  41&  60&  88& 129& 189& 277& 406& 595& 
\end{tabular}
\end{center}
\caption{The first few Narayana numbers.}
\label{nar}
\end{table}
The integer sequence $(N_i)$ (in contrast to $\nu$ and $\bf n$) has been 
well-studied under the slightly awkward name of {\it Narayana's cows sequence}.
See, for example, \cite{Allouche&Johnson:1996,Lin:2021,Crilly:2024}.
We call $(N_i)$ the {\it Narayana numbers;} they form sequence
\seqnum{A000930} in the OEIS.  The Narayana numbers can be extended
to negative indices using the recurrence $N_i = N_{i-1} + N_{i-3}$.

We can build a numeration system from $(N_i)_{i \geq 0}$, where we write every
non-negative integer as a sum of distinct Narayana numbers $N_i$.
We can write such a representation for $m$ as a binary string
$e_1 \cdots e_t$, where 
\begin{equation}
m =  \sum_{1 \leq i \leq t} e_{i} N_{t-i}.
\label{sum}
\end{equation}
The sum \eqref{sum} is denoted $[e_1 e_2 \cdots e_t]_N$.
In other words, the
representations we study here are written most-significant-digit first.
For example, $27 = N_7 + N_4 + N_1 = 19 + 6 + 2 = [10010010]_N$.

\begin{theorem}
Every positive integer $m$ has a unique representation $e_1 \cdots e_t$,
provided the leading digit $e_1$ is nonzero,
and there are no occurrences of the block $11$ or $101$ in $e_1 \cdots e_{t}$.
\end{theorem}

\begin{proof}
This result can be found as Theorem 7.1 of
\cite{Hoggatt&Bicknell-Johnson:1982}, but
it also follows from the more general
results of Fraenkel \cite{Fraenkel:1985}.
\end{proof}

The unique (canonical) Narayana representation of $m$ is written
$(m)_N = e_1 e_2 \cdots e_t$.   The representation of $0$ is $\epsilon$,
the empty string.
The first few Narayana representations are given in Table~\ref{tab2}.
\begin{table}[htb]
\begin{center}
\begin{tabular}{c|c||c|c||c|c}
$i$ & $(i)_N$ & $i$ & $(i)_N$ & $i$ & $(i)_N$ \\
\hline
0&       $\epsilon$&       7&   10001&      14& 1000001 \\
      1&       1&       8&   10010&      15& 1000010 \\
      2&      10&       9&  100000&      16& 1000100 \\
      3&     100&      10&  100001&      17& 1001000 \\
      4&    1000&      11&  100010&      18& 1001001 \\
       5&    1001&      12&  100100&      19&10000000\\
       6&   10000&      13& 1000000&      20&10000001
\end{tabular}
\end{center}
\caption{Narayana representations.}
\label{tab2}
\end{table}

Define $p_i$ to be the prefix of length $i$ of $\bf n$
and let $q_i$ be $p_{N_i} = \nu^i (0)$.
The Narayana representation of $i$ also gives a simple description
for $p_i$, as follows.
\begin{theorem}
Write the Narayana representation of $i$ in the form
$N_{d_1} + N_{d_2} + \cdots + N_{d_t}$, where
$d_1 > d_2 > \cdots > d_t$.
Then $p_i = q_{d_1} q_{d_2} \cdots q_{d_t}$.
\end{theorem}

\begin{proof}
A simple proof by induction, omitted.
\end{proof}

Recall that $|x|_a$ denotes the number of occurrences of the symbol
$a$ in the string $x$.  
\begin{proposition}
For $i \geq 0$ we have
\begin{align*}
|q_i|_0 &= N_{i-2}\\
|q_i|_1 &= N_{i-3}\\
|q_i|_2 &= N_{i-4}.
\end{align*}
\end{proposition}

\begin{proof}
A simple proof by induction, omitted.
\end{proof}

\begin{corollary}
Let the representation of $i$ be $e_1 e_2 \cdots e_j$.
Then
\begin{align*}
|p_i|_0 &= [e_1 \cdots e_{j-2}]_N + e_{j-1} + e_j \\
|p_i|_1 &= [e_1 \cdots e_{j-3}]_N + e_{j-2} + e_{j-1} \\
|p_i|_2 &= [e_1 \cdots e_{j-4}]_N + e_{j-3} + e_{j-2} .
\end{align*}
\end{corollary}

\section{Closed form and estimates for the Narayana numbers}

As is well-known from the theory of linear recurrences, there is a closed
form for the Narayana numbers of the form
\begin{equation}
N_i = c_1 \alpha^i + c_2 \beta^i + c_3 \gamma^i,
\label{binet}
\end{equation}
where $\alpha$ and $c_1$ are real, and $(\beta,\gamma)$
and $(c_2, c_3)$ are complex conjugates pairs.
Here $\alpha,\beta,\gamma$ are the zeroes of the polynomial $X^3-X^2-1$.
We therefore get, using the triangle inequality and \eqref{binet}, that
\begin{equation}
|N_i - c_1 \alpha^i| \leq  |c_2 \beta^i + c_3 \gamma^i |  
 \leq |c_2 \beta^i| + |c_3 \gamma^i| 
 = |c_2| |\beta|^i + |c_3| |\gamma|^i
 = 2|c_2| |\beta|^i .
\label{eq1}
\end{equation}

We now obtain some explicit estimates for the quantities
$N_{i+j} - \alpha^j N_i$.  These will be useful later in 
Section~\ref{kimosec}, Theorem~\ref{kimo}.

For $k \geq 0$ we have
$$ N_{i+k} - \alpha^k N_i = 
(N_{i+k} - c_1 \alpha^{i+k}) - \alpha^k (N_i - c_1 \alpha^i) ,$$
and so by the triangle inequality and \eqref{eq1} we get
\begin{align}
|N_{i+k} - \alpha^k N_i| &\leq |N_{i+k} - c_1 \alpha^{i+k}| + \alpha^k |N_i - c_1 \alpha^i| \nonumber \\
&\leq 2 |c_2| (|\beta|^k + \alpha^k ) |\beta|^i.
\label{nik}
\end{align}

Closed forms for $\alpha,\beta,\gamma, c_1, c_2, c_3$ were given by
Lin \cite{Lin:2021}, and from these we can easily compute approximations
to their values, as follows:
\begin{align*}
\alpha &\doteq 1.46557123187676802665673 \\
\beta &\doteq -0.2327856159383840133283656 + 0.7925519925154478483258983i \\
\gamma &\doteq -0.2327856159383840133283656 - 0.7925519925154478483258983i\\
|\beta| &= |\gamma| \doteq 0.826031357654186955968987\\
c_1 &= \alpha^5/(\alpha^3+2) \doteq 1.313423059852349798783263 \\
c_2 &= \beta^5/(\beta^3+2) \doteq -0.15671152992617489939163 - 0.001340333618411808095189 i \\
c_3 &= \gamma^5/(\beta^3+2) \doteq -0.15671152992617489939163 + 0.001340333618411808095189 i \\
|c_2| &= |c_3| \doteq 0.15671726167213060374568596.
\end{align*}

Now for $k = 1,2,3$, using \eqref{nik} and the numerical approximations above,
we get
\begin{align}
|N_{i+1} - \alpha N_i | & < 0.71826736534411 (0.8260313576542)^i  \label{n1}\\
|N_{i+2} - \alpha^2 N_i | &< 0.887090800406 (0.8260313576542)^i \label{n2}\\
|N_{i+3} - \alpha^3 N_i | &< 1.16331950440432 (0.8260313576542)^i . \label{n3}
\end{align}

We now prove a useful lemma.
\begin{lemma}
For $i \geq 0$ we have the bounds
\begin{align}
-0.79752082381 &< [(i)_N 0]_N - \alpha i  < 1.04861494527  \label{bnd0} \\
-1.09505816541 &< [(i)_N 00]_N - \alpha^2 i < 1.684304571609  \label{bnd00} \\
-1.10019497962 &< [(i)_N 000]_N - \alpha^3 i < 1.70593793584  \label{bnd000} .
\end{align}
\end{lemma}

\begin{proof}
To prove Eq.~\eqref{bnd0}, we start by
writing $i$ in its Narayana representation:
$i = N_{f_1} + \cdots + N_{f_s}$ for some integers
$0 \leq f_1 < f_2 < \cdots < f_s$.
Then
$$[(i)_N 0]_N = N_{f_1+1} + \cdots + N_{f_s + 1},$$
and hence
$$[(i)_N 0]_N - \alpha i 
= \sum_{1 \leq j \leq s} (N_{f_j + 1} - \alpha N_{f_j}).$$
Now we split the sum on the right-hand side 
into two pieces: one where $f_j < 30$, and one
where $f_j \geq  30$:
$$
 [(i)_N 0]_N - \alpha i  = 
\sum_{0 \leq f_j < 30} (N_{f_j + 1} - \alpha N_{f_j}) 
+
\sum_{f_j \geq 30} (N_{f_j + 1} - \alpha N_{f_j}) ,$$
and hence, by using Eq.~\eqref{n1}, we get
\begin{align}
| [(i)_N 0]_N - \alpha i  - 
\sum_{0 \leq f_j < 30} (N_{f_j + 1} - \alpha N_{f_j}) | &\leq 
\sum_{k \geq 30} |N_{k + 1} - \alpha N_{k}| \nonumber \\
& < \sum_{k \geq 30} 0.71826736534411 (0.8260313576542)^k \nonumber \\
& < 0.01335706955 . \label{a11}
\end{align}
We now compute bounds on
$ \sum_{0 \leq f_j < 30} (N_{f_j + 1} - \alpha N_{f_j}) $
by direct computation.  If we examine all $k$ in the
range $0 \leq k < N_{31} = 125491$, then we see the last
$30$ bits of all possible Narayana expansions.  By direct
calculation we get
$$-0.7841637542588 < \sum_{0 \leq f_j < 30} (N_{f_j + 1} - \alpha N_{f_j})
< 1.035257875716.$$
Putting this together with \eqref{a11} we get
\begin{equation}
-0.79752082381 < [(i)_N 0]_N - \alpha i  < 1.04861494527, \label{x795}
\end{equation}
as desired.

Eqs.~\eqref{n2} and \eqref{n3} can be proved in exactly the same way.
\end{proof}

\begin{remark}
We can easily improve the upper and lower bounds by doing  more
computation, replacing the $30$ with a large number.

Similar, but weaker numerical
results were previously obtained by Dilcher \cite{Dilcher:1993}
and Letouzey \cite{Letouzey:2025}
in, however, more generality than what we obtained here.
\end{remark}

\section{Finite automata and {\tt Walnut}}

For more about the theory of finite automata, see 
\cite{Hopcroft&Ullman:1979}, for example.

We consider sequences $(f_i)_{i \geq 0}$ defined by finite automata that
take, as inputs, the Narayana representation of one or more integers.
There are three different ways to do this, as follows:
\begin{itemize}
\item The simplest kind of definition solves the membership problem
for the sequence.  Here the automaton takes the single number $j$ as input
(in Narayana representation), and accepts if and only if there is an $i$
such that $f(i) = j$.  This is, of course, not so useful if $f(j)$ 
grows slower than $j$.

\item If the sequence $f$ takes its values in a finite alphabet, then
it can be represented as a DFAO (deterministic finite
automaton with output) that on input $i$ leads to a state with
$f(i)$ as the output.
For example, the Narayana word ${\bf n} = n_0 n_1 n_2 \cdots$
at position $i \geq 0$
takes the value $1$ if $(i)_N$ ends in $1$,
the value $2$ if $(i)_N$ ends in $10$, and $0$ otherwise.  This gives
the automaton in Figure~\ref{naa} computing $n_i$.
\begin{figure}[htb] 
\begin{center}
\includegraphics[width=4in]{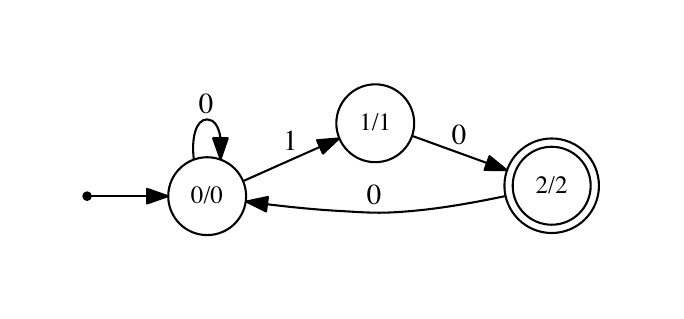}
\end{center}
\caption{Narayana automaton computing $n_i$.}
\label{naa}
\end{figure}

\item A more complicated but more useful definition takes $i$ and $j$ as
input, in parallel, and accepts if and only if $f(i) = j$.
In this case we may have to pad the shorter of the two inputs 
with leading zeroes so that both $i$ and $j$ Narayana representations
are of the same length.
Note that with this
type of definition we can easily compute both of the previous
two representations of $f$, but not necessarily vice versa.
\end{itemize}
As an example, consider the automaton
depicted in Figure~\ref{incr}; it accepts the representations of
$i$ and $j$ in parallel if and only if $j=i+1$; it is an
{\it incrementer}.
\begin{figure}[htb]
\begin{center}
\includegraphics[width=5in]{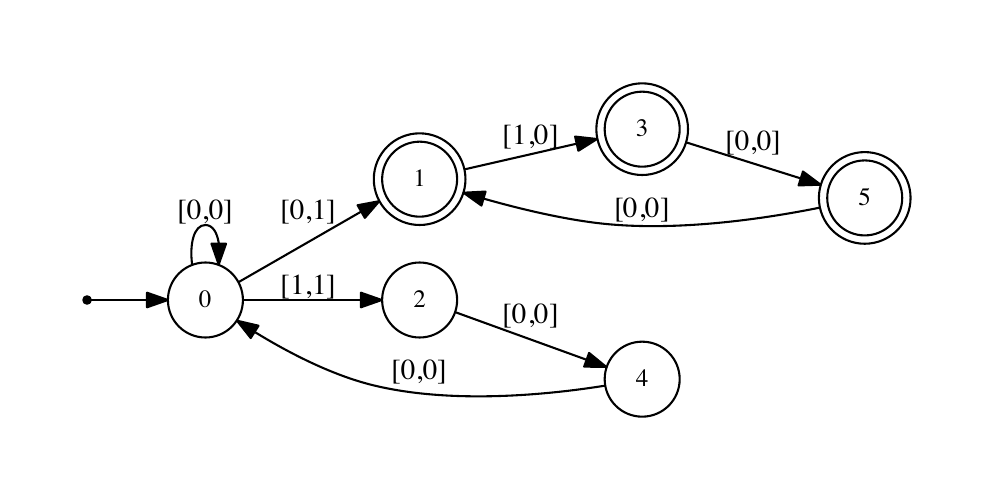}
\end{center}
\caption{Incrementer automaton for Narayana representations.}
\label{incr}
\end{figure}
The correctness of this automaton follows from the following
summation identity for the Narayana numbers, which is
easily verifiable by induction:
\begin{lemma}
Let $i \in \{1,2,3\}$ and $h \geq 0$.  Then
$$ N_i + N_{i+3} + \cdots + N_{i+3h} = N_{i+3h+1} - 1.$$
\end{lemma}
In this paper, we call all three of these varieties
{\it Narayana automata.}


Recall that we call a numeration system {\it addable} is
if there is a finite automaton recognizing the relation
among integers $x + y = z$, if $x,y,z$ are represented
in the numeration system and input in parallel, with
the shorter input padded with leading $0$'s, if necessary.

Many of the results in this paper depend crucially on the
following results.

\begin{theorem}
The Narayana numeration system is addable, and an adder automaton
exists with $250$ states.
\end{theorem}

\begin{proof}
According to results of Frougny and co-authors
\cite{Frougny:1992,Frougny&Solomyak:1992,Berend&Frougny:1994},
this follows from the fact that $\alpha$ is a Pisot number
(it is a real number $> 1$ whose conjugates $\beta$ and $\gamma$
are inside the unit circle).

We can construct the automaton using the algorithm suggested
in these papers.
\end{proof}


{\tt Walnut} is free software that can handle logical expressions
phrased in Buchi arithmetic, a form of first-order logic
with addition.  If the expression has no free variables, {\tt Walnut}
evaluates the expression as {\tt TRUE} or {\tt FALSE}, and its calculations
provide a rigorous proof of correctness of the result.  If the expression
has free variables, then {\tt Walnut} produces a finite automaton accepting
exactly the representations of the integers that make the assertion true.

For example, the automaton in Figure~\ref{incr} can be produced by
the {\tt Walnut} command
\begin{verbatim}
def incr "?msd_nara j=i+1":
\end{verbatim}

A brief guide to {\tt Walnut}'s syntax follows; for more
detail, see \cite{Shallit:2023}.  The {\tt def} command
defines an automaton for future use.  The {\tt eval} command evaluates
whether an expression with no free variables is {\tt TRUE} or {\tt FALSE}.
The indication {\tt ?msd\_nara} indicates that integers in the expression
are represented in the Narayana representation system.
The letter {\tt A} represents the universal quantifier $\forall$
and the letter {\tt E} represents the existential quantifier
$\exists$.  The logical operations are {\tt \&} (AND), 
{\tt |} (OR), {\tt \char'176} (NOT), {\tt =>} (implication),
{\tt <=>} (IFF).  {\tt NA} represents the infinite word $\bf n$,
indexed starting at position $0$.  Regular expressions can
be defined using the {\tt reg} command.

\section{Properties of the sequence $\bf n$}

In this section we use {\tt Walnut} to prove some basic properties
of the sequence $\bf n$.  One of our basic tools is a certain
Narayana automaton,
{\tt naraef}, that takes three arguments $i,j,m$ and accepts
if and only if the factors of length $m$ beginning at positions
$i$ and $j$ are the same; that is, if
${\bf N}[i..i+m-1] = {\bf N}[j..j+m-1]$.
This can be created in {\tt Walnut} as follows:
\begin{verbatim}
def naraef "?msd_nara Au,v (u>=i & u<i+m & u+j=v+i) => NA[u]=NA[v]":
# 71 states 
\end{verbatim}

\subsection{Appearance}
Recall that the appearance function $A_m$ is the smallest bound
$x$ such that every length-$m$ factor that appears,
occurs at a starting position $\leq x$.  We can compute this
for $\bf n$ using {\tt Walnut}, as follows:
\begin{verbatim}
def has_all "?msd_nara Ai Ej j<=x & $naraef(i,j,m)":
# 58 states
def app "?msd_nara $has_all(m,x) & Ay $has_all(m,y) => y>=x":
# 41 states
\end{verbatim}
Here ${\tt app}(m,x)$ returns {\tt TRUE} if $A_m = x$.

\begin{theorem}
Suppose $m \geq 2$.   Define $D_i = \lfloor (N_i-N_{i+1} + 2N_{i+2})/3 \rfloor$
for $i \geq 0$.  Then $A_m = N_{i+4} - 1$ for
$D_i < m \leq D_{i+1}$.
\end{theorem}

\begin{proof}
We can prove this with {\tt Walnut}, as follows.
\begin{verbatim}
reg five msd_nara msd_nara msd_nara msd_nara msd_nara 
   "[0,0,0,0,0]*[1,0,0,0,0][0,1,0,0,0][0,0,1,0,0]
   [0,0,0,1,0][0,0,0,0,1][0,0,0,0,0]*":
eval appearance_check "?msd_nara Am,t,d0,d1,v,w,x,y,z 
  ($five(v,w,x,y,z) & $app(m,t) & d0=((z+2*x)-y)/3 &
  d1=((y+2*w)-x)/3 & d0<m & m<=d1 & m>=2) => t+1=v":
# 82856 ms
\end{verbatim}
And {\tt Walnut} returns {\tt TRUE}.
\end{proof}

\begin{corollary}
We have $A_m/m \leq \alpha^2+\alpha \doteq 3.61347026758$ for $m\geq 1$,
and the constant is optimal.
\end{corollary}

\begin{proof}
From the description of the appearance function above,
we know that the largest value of $A_m/m$ in the
range $m \in [D_i + 1, D_{i+1}]$ is at $m = D_i + 1$.
In this case $A_m/m = (N_{i+4} - 1)/(D_i + 1)$, which 
tends to $3\alpha^4/(1-\alpha+2\alpha^2) = \alpha^2+\alpha$ from below.
\end{proof}

\subsection{Critical exponent}

Recall that a finite word $w = w[0..n]$ has period $p$, for $1 \leq p \leq n$,
if $w[0..n-p] = w[p..n]$.  For example, {\tt alfalfa} has periods
$3,6,$ and $7$.  The smallest period is called {\it the\/} period, and
is denoted $\per(w)$.  The exponent of a finite word $w$ is
$\exp(w) := |w|/\per(w)$.  The critical exponent of a (finite or infinite)
word $w$ is defined to be $\sup \{ \exp(x) \suchthat 
x \text{ a nonempty factor of } w \}$.

\begin{theorem}
The critical exponent of $\bf n$ is
$\beta = (\alpha^2 + \alpha + 5)/3 \doteq 2.871156755860518460227649$,
where
$\alpha \doteq 1.465571231876768026656731$ is the real root of $X^3=X^2 + 1$.
\label{crit}
\end{theorem}

\begin{proof}
We use {\tt Walnut}.  The idea is to successively obtain Narayana automata,
as follows:
\begin{itemize}
\item {\tt nara\_isaper}:  takes $3$ inputs $i, m, p$ and accepts
if $p$ is a period of ${\bf n}[i..i+m-1]$.

\item {\tt nara\_per}: takes $3$ inputs $i, m, p$ and accepts
if $p$ is the (smallest) period of ${\bf n}[i..i+m-1]$.

\item {\tt nara\_lp}:  takes $2$ inputs $m, p$ and accepts 
if $p$ is the smallest period over all length-$m$ factors of
$\bf n$.

\item {\tt nara\_max}:  takes two inputs $m$ and $p$ and
accepts if the shortest period over all length-$m$ words
is $p$, and every longer factor has period $>p$.

\item {\tt bignm}:  takes two inputs $m$ and $p$ and
accepts if the shortest period over all length-$m$ words
is $p$, and every longer factor has period $>p$, 
and $m/p > 14/5$.
\end{itemize}
The point, of course, is that the critical exponent
is the supremum of the set $\{ m/p \suchthat {\tt bignm}
\text{ accepts the pair $(m,p)$} \}$.
The implementation in {\tt Walnut} is as follows:
\begin{verbatim}
def nara_isaper "?msd_nara p>0 & p<=m & $naraef(i,i+p,m-p)":
# given i,m,p, decide if p>0 is a period of N[i..i+m-1]
# 882 states

def nara_per "?msd_nara $nara_isaper(i,m,p) & Aq (q<p) 
   => ~$nara_isaper(i,m,q)":
# least period
# 410 states

def nara_lp "?msd_nara Ei $nara_per(i,m,p) & 
   Aj,q $nara_per(j,m,q) => q>=p":
# least period over all words of length m
# 33 states

def nara_max "?msd_nara $nara_lp(m,p) & Aq (q>m) => ~$nara_lp(q,p)":
# there is a factor of length m with period p and every longer word
# has a longer period; 21 states

def bignm "?msd_nara $nara_max(m,p) & 5*m>14*p":
# (m,p) with m/p > 15/4
\end{verbatim}

The automaton {\tt bignm} is displayed in Figure~\ref{naramax}.
\begin{figure}[htb]
\begin{center}
\includegraphics[width=6.5in]{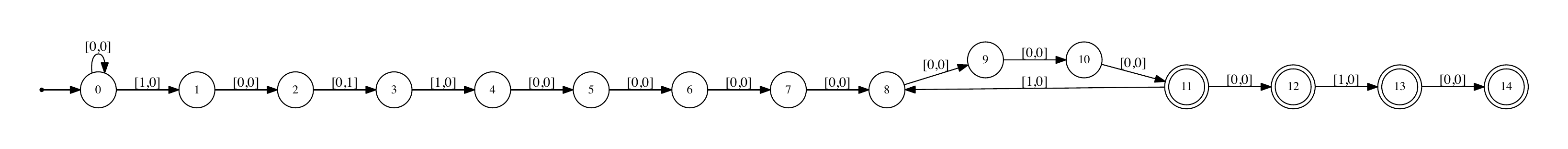}
\end{center}
\caption{Narayana automaton for maximum of period.}
\label{naramax}
\end{figure}

Inspection of Figure~\ref{naramax} indicates that $(m,p)$ is accepted
and only if it is of the form
$$ [1,0][0,0][0,1][1,0][0,0]^7([1,0][0,0]^3)^i \{ \epsilon, [0,0],
[0,0][1,0], [0,0][1,0][0,0] \}.$$
These correspond to fractions
\begin{align*}
& (N_{4n+14} + N_{4n+11} + \sum_{0\leq i \leq n} 4i+3 )/N_{4n+12} \\
& (N_{4n+15} + N_{4n+12} + \sum_{0 \leq i \leq n} 4i+4)/N_{4n+13} \\
& (N_{4n+16} + N_{4n+13} + 1 + \sum_{0 \leq i \leq n} 4i+5)/N_{4n+14} \\
& (N_{4n+17} + N_{4n+14} + 2 + \sum_{0 \leq i \leq n} 4i+6)/N_{4n+17} 
\end{align*}

\begin{lemma}
Let $n \geq 0$ and write $n = 4a+b$ with $b \in \{3,4,5,6\}$.
Define $g_b$ for $b \in \{3,4,5,6\}$ as follows:
$g_3 = 4$, $g_4 = 5$, $g_5 = 7$, $g_6 = 11$.
Then
$$N_b + N_{4+b} + N_{8+b} + \cdots + N_{4a+b} = 
(-5 N_{4a+b+11} + 4N_{4a+b+10} + 5N_{4a+b+9} - g_b)/3.$$
\label{a4b}
\end{lemma}

\begin{proof}
A tedious induction, or by {\tt Walnut}.  We omit the details.
\end{proof}

We can now complete the proof of Theorem~\ref{crit}.
By rearranging and applying the Lemma~\ref{a4b}, we find  the four fractions are
\begin{align*}
&(N_{4n+14} + N_{4n+13} + 5 N_{4n+12} - 4)/(3N_{4n+12})\\
&(N_{4n+15} + N_{4n+14} + 5 N_{4n+13} - 5)/(3N_{4n+13})\\
&(N_{4n+16} + N_{4n+15} + 5 N_{4n+14} - 4)/(3N_{4n+14})\\
&(N_{4n+17} + N_{4n+16} + 5 N_{4n+15} - 5)/(3N_{4n+15}).
\end{align*}
Using estimates for $N_{i+j}-\alpha^j N_i$ for $j=0,1,2$
we see that all of these fractions are
$< (\alpha^2 + \alpha + 5)/3$ but converge to it from below.
\end{proof}

\subsection{Subword complexity}

Recall that the subword complexity function of a word $\rho_{\bf x}$ maps $n$ to
the number of distinct factors of length $n$.
\begin{theorem}
The subword complexity function $\rho_{\bf n}$ satisfies
$\rho_{\bf n}(i) = 2i+1$.
\label{sc}
\end{theorem}

\begin{proof}
We can generate a linear representation for the number of distinct
length-$n$ factors using {\tt Walnut}.  To do so we count the number
of {\it novel\/} factors; that is, those that never appeared earlier
in $\bf n$.  We can also generate a linear representation
for the function $n \rightarrow 2n+1$ using {\tt Walnut}.
\begin{verbatim}
def novel n "?msd_nara Aj (j<i) => ~$naraef(i,j,n)":
# 82 states
def a2n1 n "?msd_nara i<2*n+1":
# 123 states
\end{verbatim}
Now we use a minimization algorithm of Berstel and Reutenauer
\cite[\S 2.3]{Berstel&Reutenauer:2011}.  When we minimize these
two linear representations, we get two identical linear
representations of rank $12$.  This proves
that the functions $\rho_{\bf n} (i)$ and $i \rightarrow 2i+1$
are identical.
\end{proof}

\begin{remark}
This is a special case of a much
more general result \cite[Theorem 1.1]{Bernat&Masakova&Pelantova:2007}.
\end{remark}

\subsection{Palindromes}

\begin{proposition}
The only nonempty palindromes in $\bf n$ are $0, 1, 2, 00, 010, 101$.
\end{proposition}

\begin{proof}
We can easily verify this by enumerating all of the factors of
length $\leq 5$.
\end{proof}

\begin{remark}
By the much more
general results in \cite{Ambroz&Masakova&Pelantova&Frougny:2006}
we know there are do not exist arbitrarily large palindromes in $\bf n$.
\end{remark}

\subsection{Right-special factors}

Recall that a factor $x$ is called right-special if
there are two distinct letters $a$ and $b$
such that $xa$ and $xb$ are both factors.

\begin{theorem}
There are two right-special factors of every length $n \geq 1$, one
one of which ends in $0$ and one that ends in $1$.  Furthermore
there is a Narayana DFAO {\tt SP0} (resp., {\tt SP1})
of $56$ states (resp., $64$ states)
that generates an infinite word such that the reverse of every
prefix is a right-special factor that ends in $0$ (resp., $1$).
\end{theorem}

\begin{proof}
The following {\tt Walnut code} checks the first claim.
\begin{verbatim}
def rtspec0 "?msd_nara $novel(x,n) & NA[x+n-1]=@0 &
   Ej $naraef(x,j,n) & NA[x+n]!=NA[j+n]":
# 56 states
# accepts n,x if x is the position of a novel factor ending in 0 
def rtspec1 "?msd_nara $novel(x,n) & NA[x+n-1]=@1 &
   Ej $naraef(x,j,n) & NA[x+n]!=NA[j+n]":
# 61 states
# accepts n,x if x is the position of a novel factor ending in 1
eval exist_rs_0 "?msd_nara An (n>=1) => Ex $rtspec0(n,x)":
eval exist_rs_1 "?msd_nara An (n>=1) => Ex $rtspec1(n,x)":
\end{verbatim}

For the second claim we can define the words as follows:
\begin{verbatim}
def rtspec00 "?msd_nara Ex $rtspec0(m+1,x) & NA[x]=@0":
def rtspec01 "?msd_nara Ex $rtspec0(m+1,x) & NA[x]=@1":
def rtspec02 "?msd_nara Ex $rtspec0(m+1,x) & NA[x]=@2":
combine SP0 rtspec00=0 rtspec01=1 rtspec02=2:
def rtspec10 "?msd_nara Ex $rtspec1(m+1,x) & NA[x]=@0":
def rtspec11 "?msd_nara Ex $rtspec1(m+1,x) & NA[x]=@1":
def rtspec12 "?msd_nara Ex $rtspec1(m+1,x) & NA[x]=@2":
combine SP1 rtspec10=0 rtspec11=1 rtspec12=2:
# check that if there is a right-special factor of length n
# at position x and one of length n+1 at position y
# then x is a suffix of y
eval check "?msd_nara An,x,y (n>=1 & $rtspec0(n,x) &
   $rtspec0(n+1,y)) => $naraef(x,y+1,n)":
\end{verbatim}
\end{proof}

The first few terms of the sequences generated by
{\tt SP0} and {\tt SP1} are given in Table~\ref{sp0tab}.
\begin{table}[H]
\begin{center}
\begin{tabular}{c|cccccccccccccccccccccccccc}
$j$ & 0 & 1& 2& 3& 4& 5& 6& 7& 8& 9&10&11&12&13&14&15&16&17&18&19\\
\hline
{\tt SP0}[j] &  0&2&1&0&1&0&0&2&1&0&0&2&1&0&2&1&0&1&0&0&\\
{\tt SP1}[j] &  1&0&0&2&1&0&2&1&0&1&0&0&2&1&0&1&0&0&2&1&\\
\end{tabular}
\end{center}
\caption{Infinite words coding the right-special factors of ${\bf n}$.}
\label{sp0tab}
\end{table}

\subsection{Positions of symbols}
\label{pos-subsec}

There are several sequences of interest related to the position of
symbols in $\bf n$.  To agree with the existing literature,
in this section we work with ${\bf n}'$, which is ${\bf n}$ but
indexed starting at position $1$:  
$$ {\bf n}' = n'_1 n'_2 n'_3 \cdots,$$
where $n'_i = n_{i-1}$ for $i \geq 1$, and we number occurrences
where the first occurrence is occurrence $1$ (not $0$).

We define
\begin{align}
\text{$p_i (j)$ is the 
position of the $j$'th occurrence of $i$ in ${\bf n}'$, 
for $i \in \{0,1,2\}$;}
\label{pij} \\
\text{$p_{02} (j)$ is 
the position of the $j$'th occurrence of a symbol in $\{0,2\}$ in ${\bf n}'$.} \label{p02p}
\end{align}

\begin{table}[H]
\begin{center}
\begin{tabular}{c|cccccccccccccccccccccccccc}
$j$ & 1& 2& 3& 4& 5& 6& 7& 8& 9&10&11&12&13&14&15&16&17&18\\
\hline
$n'_j$ & 0&1&2&0&0&1&0&1&2&0&1&2&0&0&1&2&0&0&\\
$p_0 (j)$ & 1& 4& 5& 7&10&13&14&17&18&20&23&24&26&29&32&33&35&38\\
$p_1 (j)$ & 2& 6& 8&11&15&19&21&25&27&30&34&36&39&43&47&49&52&56\\
$p_2 (j)$ & 3& 9&12&16&22&28&31&37&40&44&50&53&57&63&69&72&76&82\\
$p_{02} (j)$ & 1& 3& 4& 5& 7& 9&10&12&13&14&16&17&18&20&22&23&24&26
\end{tabular}
\end{center}
\caption{First few values of $p_0$, $p_1$, $p_2$, and $p_{02}$.}
\end{table}

Here $p_0$ is \seqnum{A202342}, $p_1$ is \seqnum{A136496},
$p_2$ is \seqnum{A381841}, and $p_{02}$ is \seqnum{A136495}.

\begin{remark}
The four sequences $p_0(j)$, $p_1(j)$,
$p_2(j)$, $p_{02} (j)$ appear in 
\cite[Table 7.1]{Hoggatt&Bicknell-Johnson:1982}
where they are called $A_j$, $B_j$, $C_j$, and $H_j$, respectively.
\end{remark}

\begin{theorem}
We have
\begin{align}
p_{02} (j) &= [ (j-1)_N 0]_N + 1 \label{p02}\\
p_0(j) &= [(j-1)_N 00]_N + 1 \label{p0}\\
p_1(j) &= [(j-1)_N 000]_N + 2 = p_0(j) + j \label{p1}\\
p_2(j) &= [(j-1)_N 0000]_N + 3 = p_1(j) + p_{02} (j) \label{p2} \\
\end{align}
for $j \geq 1$.
\label{p012}
\end{theorem}

\begin{proof}
We can prove all these assertions with {\tt Walnut}.  The idea
is to define automata {\tt p02}, {\tt p0}, {\tt p1}, {\tt p2}
implementing the formulas \eqref{p02}--\eqref{p2} above, and
and then check that the functions they compute really
do verify the definitions \eqref{pij}--\eqref{p02p}.

We need
an automaton {\tt lshift} which takes two arguments that are
binary strings, $x$ and $y$, and accepts if there is a string
$z$ such that $x = 0z$ and $y=z0$.  This is generated by
using a regular expression.

The following {\tt Walnut} code implements the formulas above.
\begin{verbatim}
reg lshift {0,1} {0,1} "([0,0]|[0,1][1,1]*[1,0])*":
def p0 "?msd_nara Ex,y $lshift(j-1,x) & $lshift(x,y) & z=y+1":
def p1 "?msd_nara Ex,y,t $lshift(j-1,x) & $lshift(x,y) & $lshift(y,t) &
   z=t+2":
def p2 "?msd_nara Ex,y,t,u $lshift(j-1,x) & $lshift(x,y) & $lshift(y,t) &
   $lshift(t,u) & z=u+3":
def p02 "?msd_nara Ex $lshift(j-1,x) & z=x+1":
\end{verbatim}

Now we check to see if they match the definition.  For example, for
$p_0$ we need to see that (1) it is an increasing function  and
(2) it takes a value $x$ iff $n'_x = 0$.
\begin{verbatim}
eval p0_test1 "?msd_nara Aj,x,y (j>=1 & $p0(j,x) & $p0(j+1,y)) => x<y":
# does p0 compute an increasing function?
eval p0_test2 "?msd_nara Ax (Ej j>=1 & $p0(j,x)) <=> NA[x-1]=@0":
# is x a value of p0 if and only if the corresponding symbol of n' equal 0?
\end{verbatim}

Similar code can be used to check the correctness of the other three
functions:
\begin{verbatim}
eval p1_test1 "?msd_nara Aj,x,y (j>=1 & $p1(j,x) & $p1(j+1,y)) => x<y":
eval p1_test2 "?msd_nara Ax (Ej j>=1 & $p1(j,x)) <=> NA[x-1]=@1":
eval p1_test3 "?msd_nara Aj,x,y (j>=1 & $p0(j,x) & $p1(j,y)) => y=x+j":
eval p2_test1 "?msd_nara Aj,x,y (j>=1 & $p2(j,x) & $p2(j+1,y)) => x<y":
eval p2_test2 "?msd_nara Ax (Ej j>=1 & $p2(j,x)) <=> NA[x-1]=@2":
eval p2_test3 "?msd_nara Ax,y,z (Ej j>=1 & $p02(j,x) & $p1(j,y) & $p2(j,z)) => z=x+y":
eval p02_test1 "?msd_nara Aj,x,y (j>=1 & $p02(j,x) & $p02(j+1,y)) => x<y":
eval p02_test2 "?msd_nara Ax (Ej j>=1 & $p02(j,x)) <=> 
   (NA[x-1]=@0|NA[x-1]=@2)":
\end{verbatim}
\end{proof} 

\subsection{The $3$-Zeckendorf array}
\label{zeck3}

We now discuss the relationship of the sequences
$p_0$, $p_1$, $p_2$, and $p_{02}$ to the $3$-Zeckendorf
array, appearing as \cite[Table 2]{Kimberling:1995b} and 
the array $\chi_3$ in \cite[Table 2]{Ericksen&Anderson:2012},
where they are indexed starting at position $0$.
As given in Ericksen and Andersen \cite{Ericksen&Anderson:2012},
the element $z_{i,0}$ in row $i$ and column $0$ of $\chi_3$ is
the $i$'th natural number whose Narayana representation ends in
the digit $1$, while the element in row $i$ and column $j$
for $j \geq 0$ is $[(z_{i,0})_N 0^j]_N$.  It follows that
$z_{i,j} = z_{i,j-1} + z_{i,j-3}$ for $j \geq 3$.  We can then
use this recurrence relation to define $z_{i,j}$ for $j \geq 3$.

Let us now show that $z_{i,0} = p_1(i+1)-1$ for $i \geq 0$.  In
order to do that, we need to see that $(z_{i,0})_N$ ends in
$1$, and furthermore every integer whose representation ends in
$1$ appears as a value of $z_{i,0}$.  We can do this as follows:
\begin{verbatim}
def col0 "?msd_nara $p1(i+1,z+1)":
reg end1 msd_nara "(0+1)*1":
eval test1 "?msd_nara Ax (Ei $col0(i,x)) <=> $end1(x)":
\end{verbatim}

Now let us define columns $1$ and $2$ using their Narayana representations.
\begin{verbatim}
def col1 "?msd_nara Ex $col0(i,x) & $lshift(x,z)":
def col2 "?msd_nara Ex $col1(i,x) & $lshift(x,z)":
\end{verbatim}
From these we can easily define columns $-3$, $-2$, and $-1$:
\begin{verbatim}
def colm1 "?msd_nara Ex,y $col2(i,x) & $col1(i,y) & z+y=x":
def colm2 "?msd_nara Ex,y $col1(i,x) & $col0(i,y) & z+y=x":
def colm3 "?msd_nara Ex,y $col0(i,x) & $colm1(i,y) & z+y=x":
\end{verbatim}

\begin{proposition}
For $i\geq 0$ we have
\begin{itemize}
\item[(a)] $z_{i,-3} = i$.
\item[(b)] $z_{i,-2} = p_{02}(i+1)$.
\item[(c)] $z_{i,-1} = p_0(i+1)$.
\item[(d)] $z_{i,0} = p_1(i+1) - 1$.
\item[(e)] $z_{i,1} = p_2(i+1) - 1$.
\item[(f)] $z_{i,2} = p_0(i+1) + p_2(i+1) - 1$.
\item[(g)] $z_{i,3} = p_0(i+1) + p_1(i+1) + p_2(i+1) - 2$.
\item[(h)] $z_{i,4} = p_0(i+1) + p_1(i+1) + 2 p_2(i+1) - 3$.
\item[(i)] $z_{i,5} = 2p_0(i+1) + p_1(i+1) + 3 p_2(i+1) - 4$.
\item[(j)] $z_{i,6} = 3p_0(i+1) + 2p_1(i+1) + 4p_2(i+1) - 6$.
\item[(k)] $z_{i,7} = 4p_0(i+1) + 3p_1(i+1) + 6p_2(i+1) - 9$.
\end{itemize}
\end{proposition}

\begin{proof}
\item[(a)-(c)]
We use the following {\tt Walnut} code:
\begin{verbatim}
eval parta "?msd_nara Ai $colm3(i,i)":
eval partb "?msd_nara Ai,x $colm2(i,x) <=> $p02(i+1,x)":
eval partc "?msd_nara Ai,x $colm1(i,x) <=> $p0(i+1,x)":
\end{verbatim}
\item[(d)] Already discussed above.
\item[(e)] We use the {\tt Walnut} code:
\begin{verbatim}
eval parte "?msd_nara Ai,x $col1(i,x) <=> (Ey $p2(i+1,y) & x+1=y)":
\end{verbatim}
\item[(f)-(k)] Follows trivially from parts (c) and (e).
\end{proof}

Note that, suitably reindexed,
\begin{itemize}
\item $(z_{i,-3})_{i \geq 0}$ is \seqnum{A001477};
\item $(z_{i,-2})_{i \geq 0}$ is \seqnum{A136495};
\item $(z_{i,-1})_{i \geq 0}$ is \seqnum{A202342};
\item $(z_{i,0})_{i \geq 0}$ is \seqnum{A020942};
\item $(z_{i,1})_{i \geq 0}$ is \seqnum{A064105};
\item $(z_{i,2})_{i \geq 0}$ is \seqnum{A064106};
\item $(z_{i,3})_{i \geq 0}$ is \seqnum{A372749};
\item $(z_{i,4})_{i \geq 0}$ is \seqnum{A372750};
\item $(z_{i,5})_{i \geq 0}$ is \seqnum{A372752};
\item $(z_{i,6})_{i \geq 0}$ is \seqnum{A372756};
\item $(z_{i,7})_{i \geq 0}$ is \seqnum{A372757}.
\end{itemize}

\begin{proposition}
For $i \geq 0$ and $j \in \Zee$ we have
$z_{i,j} = N_{j-4} p_0(i+1) + N_{j-5} p_1(i+1) + N_{j-3} p_2 (i+1) - N_{j-2}$.
\end{proposition}

\begin{proof}
Two easy inductions, one for positive indices and one for negative indices.
\end{proof}

\subsection{Some additive number theory}

Let us do a little additive number theory with the values of
$p_0$, $p_1$, $p_2$, $p_{02}$, in analogy with what has
been done with the $2$-Zeckendorf array \cite{Phunphayap,Shallit:2022}.
Let $S,T \subseteq \Enn$ be sets of natural numbers.
Then by $S+T$ we mean the set $\{ s+t \suchthat s \in S,\ t \in T \}$.
Define the sets 
\begin{align*}
P_{02} &= \{ p_{02} (j) \suchthat j \geq 1 \} \\
P_{0} &= \{ p_{0} (j) \suchthat j \geq 1 \} \\
P_{1} &= \{ p_{1} (j) \suchthat j \geq 1 \} \\
P_{2} &= \{ p_{2} (j) \suchthat j \geq 1 \} .
\end{align*}
Then
\begin{theorem}
\leavevmode
\begin{itemize}
\item[(i)] For all $n \geq 4$ we have $n \in P_{02} + P_{02}$.
\item[(ii)] For all $n \geq 17$ we have $n \in P_0 + P_0$.
\item[(iii)] For all $n$ of the form $[(100)^i 100000]_N$, $i \geq 0$, we
have $n \not\in P_1 + P_1$.
\item[(iv)] For all $n \geq 27$ we have $n \in P_1 + P_1 + P_1$.
\item[(v)] For all $n$ of the form $[1 (00)^i 1]_N$, $i \geq 1$, we have
$n \not\in P_2 + P_2$.
\item[(vi)] For all $n \geq 140$ we have $n \in P_2 + P_2 + P_2$.
\end{itemize}
\end{theorem}

\begin{proof}
We use the following {\tt Walnut} code.
\begin{verbatim}
def s02 "?msd_nara En n>=1 & $p02(n,x)":
def s0 "?msd_nara En n>=1 & $p0(n,x)":
def s1 "?msd_nara En n>=1 & $p1(n,x)":
def s2 "?msd_nara En n>=1 & $p2(n,x)":
eval two_P02 "?msd_nara An (n>=4) => Ex,y $s02(x) & $s02(y)":
eval two_P0 "?msd_nara An (n>=17) => Ex,y $s0(x) & $s0(y)":
def two_P1 "?msd_nara Ex,y n=x+y & $s1(x) & $s1(y)":
eval three_P1 "?msd_nara An (n>=27) => Ex,y,z n=x+y+z & $s1(x) &
   $s1(y) & $s1(z)":
def two_P2 "?msd_nara Ex,y n=x+y & $s2(x) & $s2(y)":
eval three_P2 "?msd_nara An (n>=140) => Ex,y,z n=x+y+z & $s2(x) &
   $s2(y) & $s2(z)":
\end{verbatim}
The proofs of parts (i), (ii), (iv), and (vi) follow directly from the
results of {\tt Walnut}.  For the proofs of (iii) and (v) we inspect the
resulting automata.  The automaton for {\tt two\_P1} is
displayed in Figure~\ref{twoP1}.  For (iii) the path labeled $(100)^i 0010$ ends
in state $15$ and hence is nonaccepting for $i \geq 1$.
The automaton for {\tt two\_P2} is too large to display here, but
for (iv)
similarly the path labeled $10000100(0^i)1$ is nonaccepting for $i \geq 1$.
\end{proof}

\begin{figure}[htb]
\begin{center}
\includegraphics[width=6.5in]{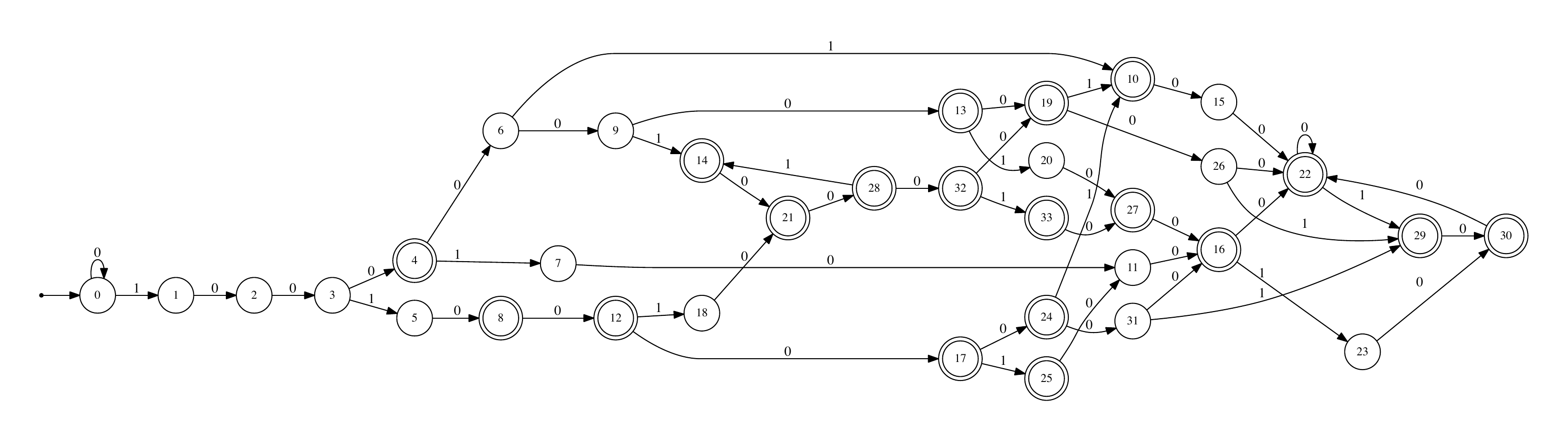}
\end{center}
\caption{Membership automaton for the set $P_1 + P_1$.}
\label{twoP1}
\end{figure}

\subsection{The Kimberling-Moses conjecture}
\label{kimosec}

Kimberling and Moses defined two increasing
sequences of positive integers $(a(j))_{j \geq 1}$ and
$(b(j))_{j\geq 1}$ that are complementary.
Here by ``complementary'' we mean these sequences have no values
in common, but considered together they take every positive integer
value exactly once.  Their classification depends on the suffix
of $(k)_N$:  
\begin{align}
\text{if $(k)_N$ ends in $1$ followed by $3s$ or $3s+2$ 
$0$'s, for $s \geq 0$, then $k$ belongs to $(a(j))_{j \geq 1}$;} \label{s32}\\
\text{if $(k)_N$ ends in $1$ followed by $3s+1$ $0$'s, for $s \geq 0$,
then $k$ belongs to $(b(j))_{j\geq 1}$. } \label{s012}
\end{align}
This is not quite
stated explicitly in their
paper, but follows immediately from the definition at the top of
p.~163 of \cite{Kimberling&Moses:2010}, and the discussion
in Section 3 of \cite{Kimberling:1995b}.
The sequence $(a(j))_{j \geq 1}$ is sequence
\seqnum{A136495} in the OEIS and the sequence $(b(j))_{j\geq 1}$
is sequence \seqnum{A136496}.
The first few values of these two sequences are given in 
Table~\ref{abtab}.
\begin{table}[htb]
\begin{center}
\begin{tabular}{c|ccccccccccccccccccccccccccccccccccc}
$j$ & 1& 2& 3& 4& 5& 6& 7& 8& 9&10&11&12&13&14&15&16&17&18&19\\
\hline
$a(j)$ & 1& 3& 4& 5& 7& 9&10&12&13&14&16&17&18&20&22&23&24&26&28\\
$b(j)$ & 2& 6& 8&11&15&19&21&25&27&30&34&36&39&43&47&49&52&56&60
\end{tabular}
\end{center}
\caption{The sequences $a(j)$ and $b(j)$.}
\label{abtab}
\end{table}

Note that the sequences $a(j)$ and $b(j)$ are {\it not\/} Beatty sequences,
although they are quite close to
$\lfloor \alpha j \rfloor$ and
$\lfloor \alpha^3 j \rfloor$, respectively.  So although they share the same
name as the functions studied in \cite[\S 4]{Ballot:2017}, they are
not the same.

Kimberling and Moses proved various results about these complementary
sequences, and also conjectured an inequality governing their asymptotic
behavior.  We prove their conjecture in what follows.

First we prove the following:
\begin{theorem}
Let $i \in \{0,1,2 \}$.
Then $k$ belongs to the sequence $(p_i (j))_{j \geq 1}$
if and only if there exists $s \geq 0$ such that
$(k)_N$ ends in $1$ followed by $3s+i$ $0$'s.
\label{cols}
\end{theorem}

\begin{proof}
We use the following {\tt Walnut} code:
\begin{verbatim}
reg end30 {0,1} "(0|1)*1(000)*":
reg end31 {0,1} "(0|1)*1(000)*0":
reg end32 {0,1} "(0|1)*1(000)*00":
eval check_p0 "?msd_nara Ak (k>=1) => ((Ej j>=1 & $p0(j,k)) <=> $end30(k))":
eval check_p1 "?msd_nara Ak (k>=1) => ((Ej j>=1 & $p1(j,k)) <=> $end31(k))":
eval check_p2 "?msd_nara Ak (k>=1) => ((Ej j>=1 & $p2(j,k)) <=> $end32(k))":
\end{verbatim}
And {\tt Walnut} returns {\tt TRUE} for all three.
\end{proof}
It turns out (and the observant reader will have already
noticed!) that the sequences $(p_{02} (j))_{j \geq 1}$ and
$(p_1 (j))_{j \geq 1}$ correspond, respectively, to the
sequences $a(j)$ and $b(j)$ in \cite{Kimberling&Moses:2010}.
We now prove this.
\begin{corollary}
We have
$a(j) = p_{02} (j)$ and $b(j) = p_1 (j)$ for $j \geq 1$.
\end{corollary}

\begin{proof}
This follows immediately from combining 
Eqs.~\eqref{s012} and \eqref{s32} with Theorem~\ref{cols}.
\end{proof}
The automata for $a(j) = p_{02} (j)$ and
$b(j) = p_1 (j)$ are given in Figures~\ref{afig} and \ref{bfig}.
\begin{figure}[htb]
\begin{center}
\includegraphics[width=6.5in]{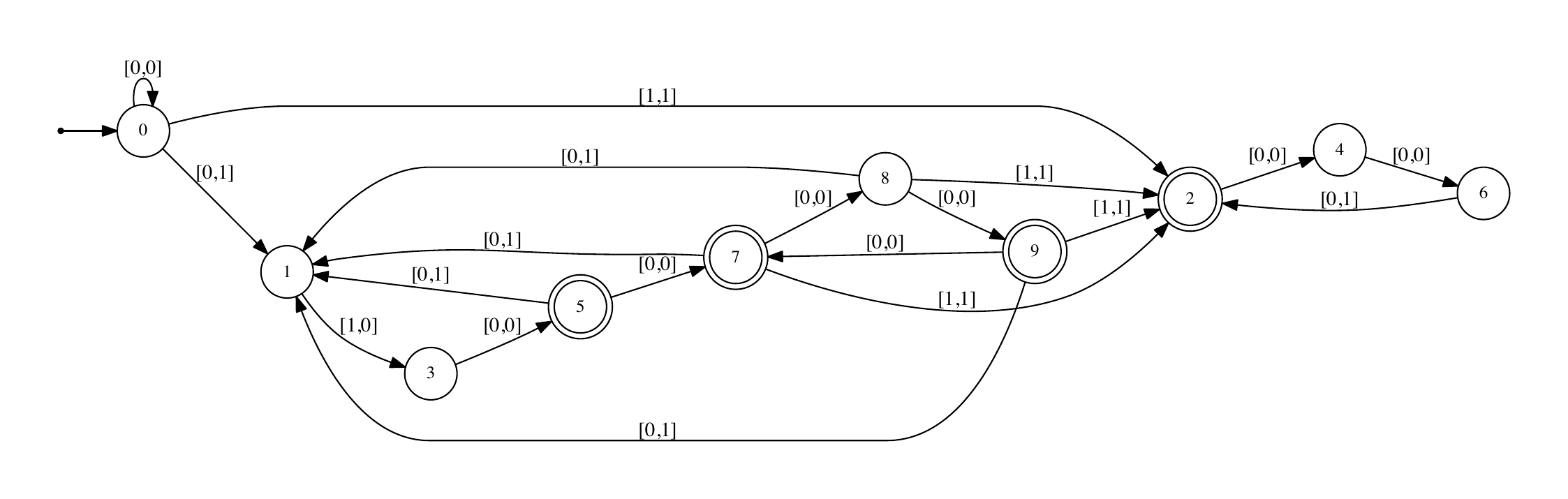}
\end{center}
\caption{Narayana automaton computing $a(i)$.}
\label{afig}
\end{figure}
\begin{figure}[htb]
\begin{center}
\includegraphics[width=6.5in]{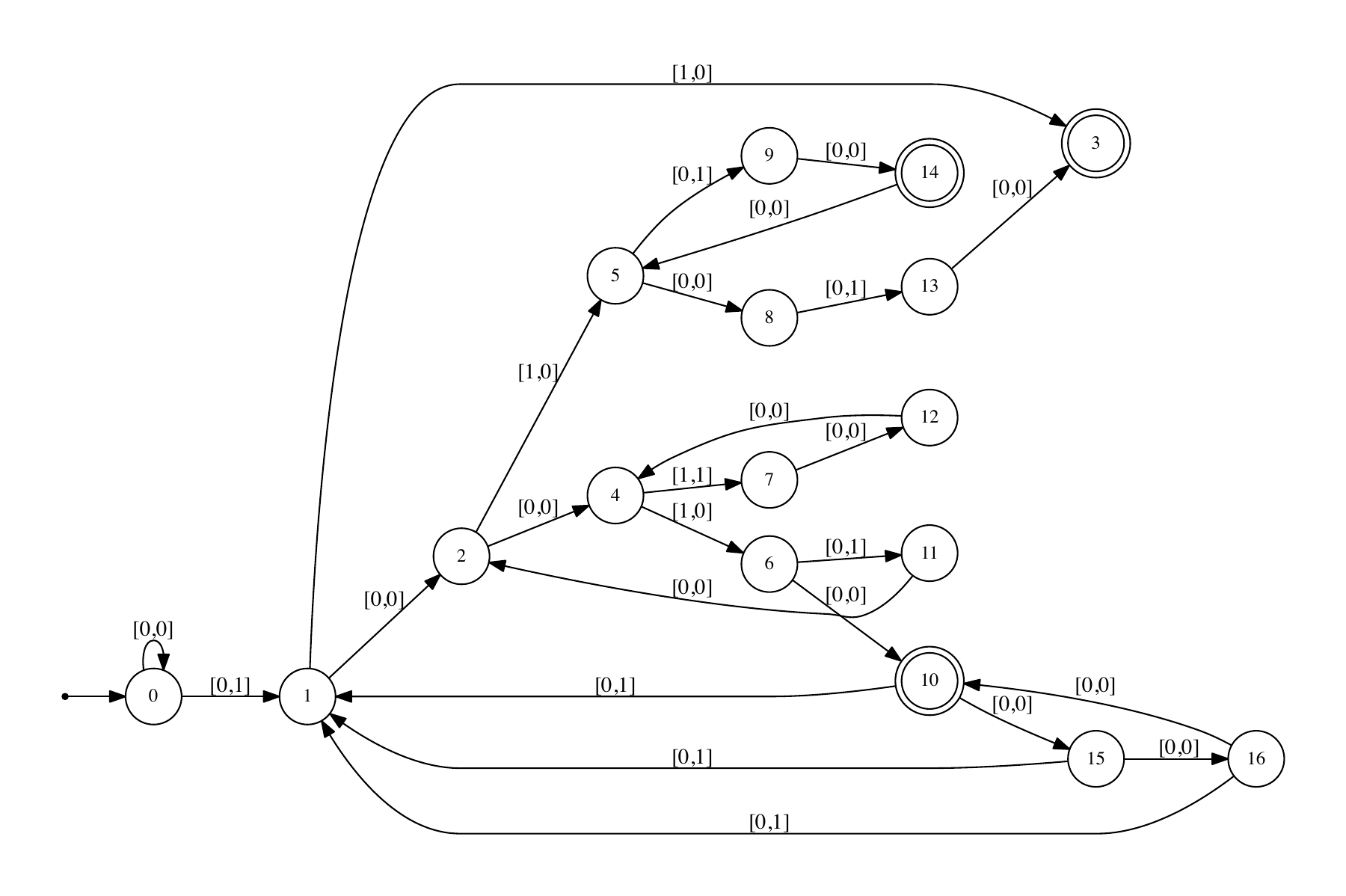}
\end{center}
\caption{Narayana automaton computing $b(i)$.}
\label{bfig}
\end{figure}

Furthermore, we can reprove the following relations
of Kimberling and Moses, linking the
sequences $a$ and $b$.
\begin{theorem}
We have
\begin{itemize}
\item[(i)] $b(j) = a(a(a(j))) + 1$;
\item[(ii)] $b(j) = a(a(j)) + j$;
\item[(iii)] $b(a(j)) = a(a(a(j))) + a(j)$;
\item[(iv)] $a(b(j)) = a(j) + b(j)$;
\item[(v)] $b(a(j)) = a(j) + b(j) - 1$.
\end{itemize}
\end{theorem}

\begin{proof}
We can use the following {\tt Walnut} code.
\begin{verbatim}
def a "?msd_nara $p02(j,x)":
def b "?msd_nara $p1(j,x)":
def item_i "?msd_nara Aj,x,y,z,w (j>=1 & $a(j,x) & $a(x,y) & $a(y,z) &
   $b(j,w)) => w=z+1":
def item_ii "?msd_nara Aj,x,y,z (j>=1 & $a(j,x) & $a(x,y) & $b(j,z)) 
   => z=y+j":
def item_iii "?msd_nara Aj,x,y,z,w (j>=1 & $a(j,x) & $a(x,y) & $a(y,z) &
   $b(x,w)) => w=z+x":
def item_iv "?msd_nara Aj,x,y,z (j>=1 & $a(j,x) & $b(j,y) & $a(y,z)) =>
   z=x+y":
def item_v "?msd_nara Aj,x,y,z (j>=1 & $a(j,x) & $b(j,y) & $b(x,z)) =>
   z+1=x+y":
\end{verbatim}
And {\tt Walnut} returns {\tt TRUE} for all of them.
\end{proof}

We now prove two inequalities conjectured by Kimberling and Moses:
\begin{theorem}
We have 
\begin{itemize}
\item[(i)] We have $-1.2630921 < a(i) - \alpha i < 0.58304372$ for all $i \geq 1$.
\item[(ii)] We have $-2.2480941 < b(i) - \alpha^3 i < 0.558039$ for all $i \geq 1$.
\end{itemize}
\label{kimo}
\end{theorem}

\begin{proof}
We use the strategy outlined in \cite{Dekking&Shallit&Sloane:2020}.

Recall that \eqref{x795} says that
$$
-0.79752082381 < [(i)_N 0]_N - \alpha i  < 1.04861494527.
$$
Now we use the fact that
$a(i) = p_{02} (i) = [(i-1)_N 0]_N + 1$.
Substituting in the above inequality, we get
$$ -0.79752082381 <  [(i-1)_N 0]_N - \alpha(i-1)  < 1.04861494527,$$
so
$$ -0.79752082381  <  (a(i) - 1) - \alpha(i-1)  < 1.04861494527,$$
and adding $1-\alpha$ to both sides finally gives
$$ -1.2630921 < a(i) - \alpha i < 0.58304372 .$$

\medskip

(ii) Exactly the same argument applies here.
From \eqref{bnd000} we get the result
$$ -1.10019497962 < [(i-1)_N 000]_N - \alpha^3 (i-1) < 1.70593793584 ,$$
and adding
$2- \alpha^3$ to both sides, we get
the desired result.
\end{proof}

\begin{remark}
A weaker inequality follows
from \cite[Theorem 1]{Dilcher:1993}.
\end{remark}

\subsection{Abelian properties}

We now study abelian properties of factors of $\bf n$.
Define the Parikh vector $\psi(w) = (|w|_0, |w|_1, |w|_2)$.
We say $x$ is an abelian $k$'th power if
$x = x_1 x_2 \cdots x_k$ for words $x_1, \ldots, x_k$,
where $\psi(x_1) = \psi(x_j)$ for $2 \leq j \leq k$.
The {\it order} of the abelian power is the length of
$x_1$.
For example, {\tt reappear} is an abelian square of order $4$.

Let us create automata {\tt pcount0}, {\tt pcount1}, 
{\tt pcount2} that count
the number of $0$'s (resp., $1$'s, $2$'s) 
in the prefix of $\bf n$ of length $i$:
\begin{verbatim}
def pcount0 "?msd_nara (i=0&x=0) | (i>=1 & Ey,z $p0(x,y) &
   $p0(x+1,z) & i>=y & i<z)":
def pcount1 "?msd_nara (i=0&x=0) | (i>=1 & Ey,z $p1(x,y) &
   $p1(x+1,z) & i>=y & i<z)":
def pcount2 "?msd_nara (i=0&x=0) | (i>=1 & Ey,z $p2(x,y) &
   $p2(x+1,z) & i>=y & i<z)":
\end{verbatim}

From these we can create automata that count the 
number of $0$'s (resp., $1$'s, $2$'s)
in ${\bf n}[i..i+m-1]$.
\begin{verbatim}
def count0 "?msd_nara Ex,y $pcount0(i,x) & $pcount0(i+m,y) 
   & z+x=y":
def count1 "?msd_nara Ex,y $pcount1(i,x) & $pcount1(i+m,y) 
   & z+x=y":
def count2 "?msd_nara Ex,y $pcount2(i,x) & $pcount2(i+m,y) 
   & z+x=y":
\end{verbatim}

Finally, we create an automaton that takes $i,j,m$ as input and
accepts if and only if $\psi({\bf n}[i..i+m-1]) = \psi({\bf n}[j..j+m-1])$.
\begin{verbatim}
def abeleq "?msd_nara Ex,y,z $count0(i,m,x) & $count0(j,m,x)
   & $count1(i,m,y) & $count1(j,m,y) & $count2(i,m,z) &
   $count2(j,m,z)":
\end{verbatim}

\begin{theorem}
There are abelian squares of all orders in $\bf n$.
\end{theorem}

\begin{proof}
We use {\tt Walnut}:
\begin{verbatim}
eval absquare "?msd_nara Am Ei $abeleq(i,i+m,m)":
\end{verbatim}
and it returns {\tt TRUE}.
\end{proof}

\begin{theorem}
There is a Narayana automaton of $117662$ states
that accepts the representation of those $m$ for which
there is an abelian cube factor of $\bf n$ of order $m$.
There are abelian cubes of arbitrarily large orders and
there are arbitrarily large orders with no abelian cubes.
\end{theorem}

\begin{proof}
We use {\tt Walnut}:
\begin{verbatim}
def abscube "?msd_nara Ei $abeleq(i,i+m,m) & $abeleq(i,i+2*m,m)":
# 117662 states - 234256ms
reg examples msd_nara "0*(100)*1001":
eval large_abelian_cubes "?msd_nara An $examples(n) => $abscube(n)":
reg counterex msd_nara "0*(100)*00001":
eval large "?msd_nara An $counterex(n) => ~$abscube(n)":
\end{verbatim}

\end{proof}

\begin{remark}
There are abelian cubes in $\bf n$ of the following
orders:  $3, 4, 5, 6, 7, 9, 10, 13, 15, 17, 18, 19,\ldots$.
There are no abelian cubes of the following orders:
$1, 2, 8, 11, 12, 14, 16, 20, 21, 27, \ldots $.
\end{remark}







\subsection{Balance}

Recall that a sequence is $k$-balanced if for all equal-length factors
$x$ and $y$ and all letters $a$ we have $\left| |x|_a - |y|_a \right|
\leq k$.
Clearly $\bf n$ is not $2$-balanced, which we see by 
looking at the factors $00120010120010$ and $12012001012012$.

\begin{theorem}
The sequence $\bf n$ is $3$-balanced.
\end{theorem}

\begin{proof}
We use the following {\tt Walnut} code.
\begin{verbatim}
def bal30 "?msd_nara Ai,j,n,x,y ($count0(i,n,x) & $count0(j,n,y)) 
   => (x>=y+3|y<=x+3)":
def bal31 "?msd_nara Ai,j,n,x,y ($count1(i,n,x) & $count1(j,n,y)) 
   => (x>=y+3|y<=x+3)":
def bal32 "?msd_nara Ai,j,n,x,y ($count2(i,n,x) & $count2(j,n,y)) 
   => (x>=y+3|y<=x+3)":
\end{verbatim}
All commands return {\tt TRUE}.
\end{proof}



\section{Other sequences based on Narayana representation}

Several sequences previously discussed in the literature can
be defined in terms of the Narayana representation of $i$.

For example, let the Narayana representation of $i$
be $e_1\cdots e_t$.  Then define
\begin{equation}
h(i) = [e_1 \cdots e_{t-1}]_N + e_t. \label{heq}
\end{equation}
The first few values of this sequence are given in Table~\ref{htab}.
\begin{table}[H]
\begin{center}
\begin{tabular}{c|cccccccccccccccccccccccccc}
$j$ & 0 & 1& 2& 3& 4& 5& 6& 7& 8& 9&10&11&12&13&14&15&16&17&18&19\\
\hline
$h(j)$ & 0& 1& 1& 2& 3& 4& 4& 5& 5& 6& 7& 7& 8& 9&10&10&11&12&13&13
\end{tabular}
\end{center}
\caption{First few values of the $h$ sequence.}
\label{htab}
\end{table}

We can create an automaton to compute $h$ as follows.
We need
an automaton {\tt rshift} which takes two arguments that are
binary strings, $x$ and $y$, and accepts if $y$ is
the number given by the representation of $x$, dropping the last bit.
It can be created in {\tt Walnut} using a regular expression.

We also need an automaton to decide if the last bit is $1$.
Again, we can compute this using a regular expression.
\begin{verbatim}
reg rshift {0,1} {0,1} "([0,0]|[1,0][1,1]*[0,1])*(()|[1,0][1,1]*)":
reg lastbit1 msd_nara "(0|1)*1":
def h "?msd_nara Ex $rshift(i,x) & ((z=x+1 & $lastbit1(i)) |
   (z=x & ~$lastbit1(i)))":
\end{verbatim}

This gives the automaton in Figure~\ref{hpic}.
\begin{figure}[htb]
\begin{center}
\includegraphics[width=6.5in]{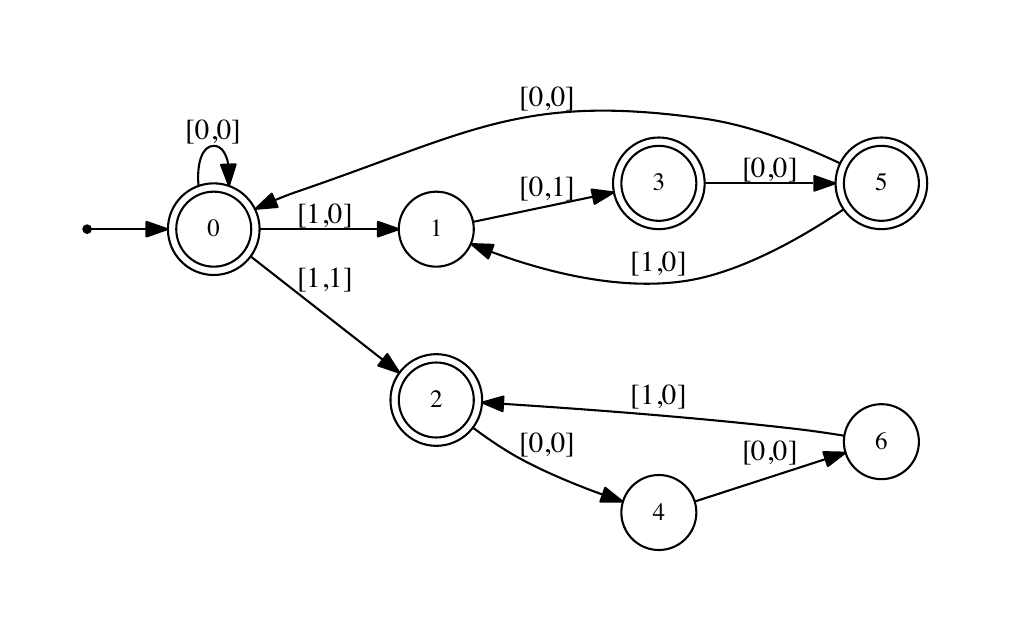}
\end{center}
\caption{Narayana automaton computing $h(i)$.}
\label{hpic}
\end{figure}

When we look up this sequence in the OEIS, we find an apparent match
with sequence \seqnum{A005374}, the so-called Hofstadter $H$-sequence
(see \cite[p.~137]{Hofstadter:1979}).
This sequence is defined by the relation $H(0) = 0$ and
$H(i) = i - H(H(H(i-1)))$.  
\begin{proposition}
We have $h(i) = H(i)$ for all $i \geq 0$.
\label{hprop}
\end{proposition}
\begin{proof}
\leavevmode
\begin{verbatim}
eval hcheck "?msd_nara Ai,x,y,z,w (i>=1 & $h(i-1,x) & $h(x,y) &
   $h(y,z) & $h(i,w)) => i=w+z":
\end{verbatim}
And {\tt Walnut} returns {\tt TRUE}.
\end{proof}

\begin{remark}
Proposition~\ref{hprop} was proven, in much greater
generality, by Meek and Van Rees \cite{Meek&VanRees:1984}.
\end{remark}

In 2002, Benoit Cloitre \cite[\seqnum{A005374}]{Sloane:2025}
conjectured an inequality for $H(i)$, 
which we can now prove using Narayana representations.
\begin{theorem}
We have $H(i) - \lfloor \alpha^{-1} i \rfloor
\in  \{ 0, 1 \}$.   
\end{theorem}

\begin{proof}
There are two cases:  when $(i)_N$ ends in $0$ and when it ends in $1$.

\medskip

$(i)_N$ ends in $0$:  Write $(i)_N = (j)_N 0$ for some $j$.
Thus $h(i) = j$.
Recall that \eqref{x795} says
$-0.79752082381 < [(j)_N 0]_N - \alpha j < 1.04861494527$.
Divide by $-\alpha$ to get
$$ -0.7154992 < j - \alpha^{-1} [(j)_N 0]_N < 0.5441707.$$
Now replace $j$ by $h(i)$ and $[(j)_N 0]_N$ by $i$, to get
\begin{equation}
-0.7154992 < h(i) - \alpha^{-1} i < 0.5441707.
\label{case1}
\end{equation}

\medskip

$(i)_N$ ends in $1$: Write $(i)_N = (j)_N 1$ for some $j$.
Thus $h(i) = j+1$.
Add $1$ to Eq.~\eqref{x795} to get
$$ 0.20247917 < [(j)_N 1]_N  - \alpha j < 2.04861495;$$
divide by $-\alpha$ to get
$$ -1.397827 < j - \alpha^{-1} [(j)_N 1]_N < -0.13815717.$$
Now replace $j$ by $h(i)-1$ and $[(j)_N 1]_N$ by $i$  and add $1$ to get
\begin{equation}
 -0.397827 < h(i) - \alpha^{-1} i < 0.86184283.
\label{case2}
\end{equation}

Putting together Eqs.~\eqref{case1} and \eqref{case2}, we get
\begin{equation}
-0.7154992 < h(i) - \alpha^{-1} i < 0.86184283.
\label{bounds1}
\end{equation}
Write $\alpha^{-1} i = \lfloor \alpha^{-1} i \rfloor + r$,
where $0 \leq r < 1$. Then substituting gives
$$ -0.7154992 < h(i) - (\lfloor \alpha^{-1} i \rfloor + r) < 0.86184283 ,$$
and adding $r$ gives
$$  -0.7154992 \leq r-0.7154992 < h(i) - \lfloor \alpha^{-1} i \rfloor 
< r+0.86184283 < 1.86184283 .$$
Since $h(i) - \lfloor \alpha^{-1} i \rfloor$ is an integer, this
gives $h(i) - \lfloor \alpha^{-1} i \rfloor \in \{ 0,1 \}$,
as desired.
\end{proof}

\begin{remark}
Dilcher \cite{Dilcher:1993} obtained an inequality similar
to \eqref{bounds1}, but weaker.
Cloitre's conjecture has recently been proven, independently,
by Letouzey \cite{Letouzey:2025}.
\end{remark}

\begin{remark} 
Cloitre conjectures in \seqnum{A082401} that
$\sum_{1 \leq i \leq t} (h(i)- \lfloor \alpha^{-1} i \rfloor) \sim
Ct$ for some constant $C$, approximately $0.53$.
The methods in this paper do not seem strong enough to prove this.
\end{remark}



\begin{proposition}
The function $H$ obeys the rule $H(i) = i - H(H(H(i-1)))$.
\end{proposition}

\begin{proof}
We use the {\tt Walnut} command
\begin{verbatim}
eval testid "?msd_nara Ai,x,y,z,w (i>=1 & $h(i,w) & $h(i-1,x) &
   $h(x,y) & $h(y,z)) => w+z=i":
\end{verbatim}
which returns {\tt TRUE}.
\end{proof}

We now examine some sequences related to $H$.  Examination of
Table~\ref{htab} suggests that
\begin{itemize}
\item[(a)] Every positive integer appears as a value of $H$;
\item[(b)] No positive integer appears $3$ or more times as a value
of $H$.
\end{itemize}
We can prove this as follows:
\begin{verbatim}
def every "?msd_nara An Ex $h(x,n)":
# every natural number appears
def nothree "?msd_nara ~En,x,y,z x<y & y<z & $h(n,x) & $h(n,y) & $h(n,z)":
# no number appears three times
\end{verbatim}

Sequence \seqnum{A202340} counts the number of times $i\geq 0$ appears as a value
of $H$.  We can find a DFAO {\tt S} for it as follows:
\begin{verbatim}
def two "?msd_nara Ex,y x<y & $h(x,n) & $h(y,n)":
# automaton for two appearances
def one "?msd_nara ~$two(n)":
combine S two=2 one=1:
\end{verbatim}
The resulting DFAO appears in Figure~\ref{figS}.
\begin{figure}[htb]
\begin{center}
\includegraphics[width=6in]{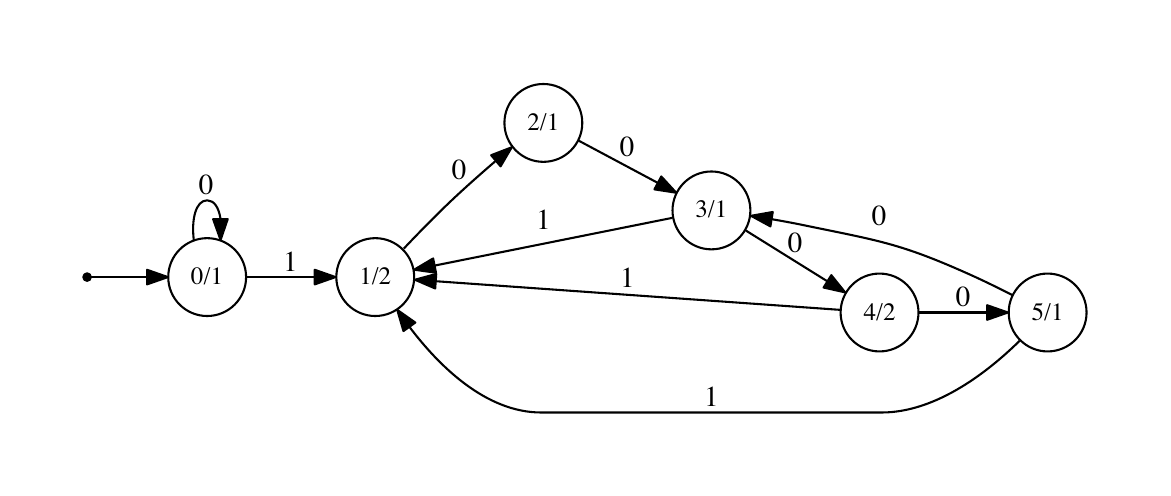}
\end{center}
\caption{Narayana automaton {\tt S} for A202340.}
\label{figS}
\end{figure}

Sequence \seqnum{A202341}, namely $0,2,3,6, 8,9,\ldots$
consists of those $i$ for which
$S[i]=1$, and \seqnum{A202342}, namely
$1, 4,5,7,10, \ldots$ consists of those $i$ for which
$S[i]=2$.
We can trivially construct a $6$-state automaton {\tt one} (resp., {\tt two})
determining membership in \seqnum{A202341} (resp., \seqnum{A202342})
from the automaton for $S$:
\begin{verbatim}
def one "?msd_nara S[i]=@1":
def two "?msd_nara S[i]=@2":
\end{verbatim}
If, however, we want a {\it synchronized automaton\/}
that accepts $i$ and $y$ in parallel and accepts
if and only if $y$ is the $i$'th term of the appropriate
sequence, there is no obvious way to get this from $S$.
Instead, for both, we {\it guess\/} the automaton and then verify its
correctness.  (Note that our indexing begins at $0$, in contrast
to the way the sequences are indexed in the OEIS; this results
in slightly simpler automata.)
\begin{table}[H]
\begin{center}
\begin{tabular}{c|cccccccccccccccccccccccccc}
$i$ & 0& 1& 2& 3& 4& 5& 6& 7& 8& 9&10&11&12&13&14&15&16\\
\hline
\seqnum{A202340}($i$) & 1& 2& 1& 1& 2& 2& 1& 2& 1& 1& 2& 1& 1& 2& 2& 1& 1\\
\seqnum{A202341}($i+1$) & 0& 2& 3& 6& 8& 9&11&12&15&16&19&21&22&25&27&28&30 \\
\seqnum{A202342}($i+1$) & 1& 4& 5& 7&10&13&14&17&18&20&23&24&26&29&32&33&35 
\end{tabular}
\end{center}
\caption{First few values of \seqnum{A202340}, \seqnum{A202341},
\seqnum{A202342}.}
\end{table}

The correctness of {\tt a202341} and {\tt a202342} is verified as follows:
\begin{verbatim}
eval increasing1 "?msd_nara An,x,y ($a202341(n,x) & $a202341(n+1,y)) => x<y":
eval correct1 "?msd_nara An (Ex $a202341(x,n)) <=> $one(n)":
eval increasing2 "?msd_nara An,x,y ($a202342(n,x) & $a202342(n+1,y)) => x<y":
eval correct2 "?msd_nara An (Ex $a202342(x,n)) <=> $two(n)":
\end{verbatim}

The Narayana automaton {\tt a202341} has 44 states,
too large to show here.  But
the Narayana automaton {\tt a202342} has only 6 states and is depicted in
Figure~\ref{a202}.
\begin{figure}[htb]
\begin{center}
\includegraphics[width=6in]{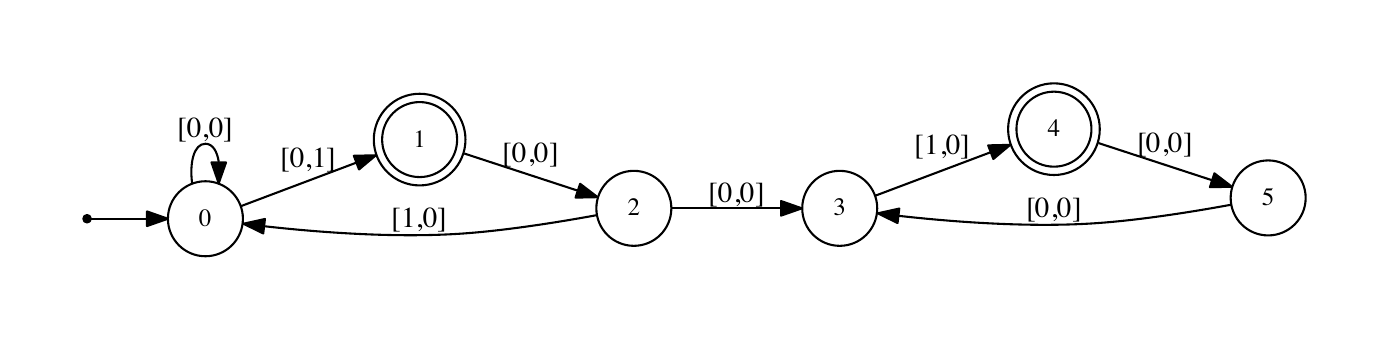}
\end{center}
\caption{Narayana automaton {\tt a202342}.}
\label{a202}
\end{figure}

The reader may have already noticed that ${\tt a202342}$ is essentially the
same as the sequence $p_0(n)$.  We can prove this as follows:
\begin{verbatim}
eval check_a "?msd_nara An,x (n>=1) => ($p0(n,x) <=> $a202342(n-1,x))":
\end{verbatim}
Using this observation, we can now prove a conjecture of Sean
Irvine stated in \seqnum{A020942}.
\begin{theorem}
We have $\seqnum{A202342}(n) + n-1 = \seqnum{A020942}(n)$ for $n \geq 1$.
\end{theorem}

\begin{proof}
As observed previously in Section~\ref{zeck3},
we have $\seqnum{A020942}(n) = p_1(n) - 1$
for $n \geq 1$.
Hence the following {\tt Walnut} code verifies the claim.
\begin{verbatim}
eval irvine "?msd_nara An,x,y (n>=1 & $p0(n,x) & $p1(n,y)) => x+n=y":
\end{verbatim}
\end{proof}

\begin{theorem}
The critical exponent of the sequence \seqnum{A202340} is
$\beta = (\alpha^2 + \alpha + 5)/3 \doteq 2.871156755860518460227649$,
where
$\alpha \doteq 1.465571231876768026656731$ is the real root of $X^3=X^2 + 1$.
\end{theorem}

\begin{proof}
Can be proved using the same technique as we used for
Theorem~\ref{crit}.
\begin{verbatim}
def sef "?msd_nara Au,v (u>=i & u<i+n & u+j=v+i) => S[u]=S[v]":
# 171 states
def s_isaper "?msd_nara p>0 & p<=m & $sef(i,i+p,m-p)":
# 1628 states
def s_per "?msd_nara $s_isaper(i,m,p) & Aq (q<p) => ~$s_isaper(i,m,q)":
# 882 states
def s_lp "?msd_nara Ei $s_per(i,m,p) & Aj,q $s_per(j,m,q) => q>=p":
# 50 states
def s_max "?msd_nara $s_lp(m,p) & Aq (q>m) => ~$s_lp(q,p)":
# 32 states
def s_bignm "?msd_nara $s_max(m,p) & 5*m>14*p":
\end{verbatim}
This gives an automaton of $27$ states, for which we can apply
the same techniques as before.
\end{proof}

\begin{theorem}
The subword complexity of the sequence \seqnum{A202340} is
$2n$ for $n \geq 1$.
\end{theorem}

\begin{proof}
Can be proved using the same technique as we used for
Theorem~\ref{sc}.
\begin{verbatim}
def novel_s n "?msd_nara Aj (j<i) => ~$sef(i,j,n)":
# 137 states
def a2n n "?msd_nara (i=0&n=0)|(n>0&i<2*n)":
# 127 states
\end{verbatim}
We then subtract the resulting linear representations, minimize
them, and verify the result is $0$.
\end{proof}

Another interesting sequence is the location of the first occurrence
of $j\geq 0$ in the sequence $(H(i))_{i \geq 0}$, that is,
$0,1,3,4,5,7,\ldots$.  We can define a synchronized automaton
for this sequence as follows:
\begin{verbatim}
def firstocc "?msd_nara $h(x,n) & ~$h(x-1,n)":
\end{verbatim}
The reader may have already guessed that this sequence is
essentially the same as $(a(n))_{n \geq 1}$ discussed
in Section~\ref{pos-subsec}.
We can prove this as follows:
\begin{verbatim}
eval check_same "?msd_nara An,x (n>=1) => ($firstocc(n,x) <=> $a(n,x))":
\end{verbatim}

\section{The Allouche-Johnson infinite word}

Allouche and Johnson \cite{Allouche&Johnson:1996} defined a certain
infinite word ${\bf aj} = 011110 \cdots$ as the limit of a
sequence of infinite words:  $X_{-2} = X_{-1} = X_0 = 0$
and $X_n = X_{n-1} \overline{X_{n-3}}$ for $n \geq 1$.
Here the overline represents binary complement:
$\overline{0} = 1$, $\overline{1} = 0$.
In this section we prove some properties of this word.

It is easy to verify that ${\bf aj}$ is the parity of the number of 
$1$'s in the Narayana representation of $n$.  We can therefore 
compute a DFAO for it, as follows:
\begin{verbatim}
reg odd1 {0,1} "0*(10*10*)*10*":
def ja "?msd_nara $odd1(n) & n>=0":
combine JA ja=1:
\end{verbatim}
It is depicted in Figure~\ref{japic}.
\begin{figure}[htb] 
\begin{center}
\includegraphics[width=5.5in]{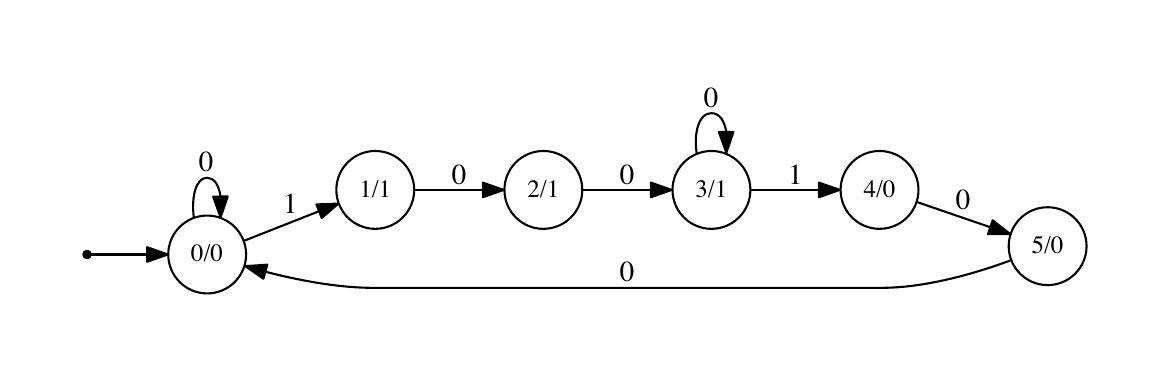}
\end{center}
\caption{Narayana DFA computing $\bf ja$.}
\label{japic}
\end{figure}

We need to be able to determine if two equal-length
factors are equal.
We'd like to compute this using the usual
\begin{verbatim}
def jaef "?msd_nara Au,v (u>=i & u<i+n & u+j=v+i) => JA[u]=JA[v]":
\end{verbatim}
but it fails because of array overflow
after {\tt Walnut} creates an automaton of more than 100 million states.

Instead we ``guess'' the automaton {\tt jaef} from empirical data,
associating states with strings over $\{0,1\}^3$ and saying two
states $p,q$ are equivalent if the result on $px$ and $qx$ agree
for all strings of length $\leq 7$.  This automaton has 258 states.

We now verify that our guessed automaton is correct.  We can do
this by checking that
\begin{itemize}
\item[(a)] ${\tt jaef}(i,j,0)$ is {\tt TRUE} for all $i$ and $j$;
\item[(b)] if ${\tt jaef}(i,j,n)$ holds, then ${\tt jaef}(i,j,n+1)$ is
{\tt TRUE} iff ${\tt JA}[i+n] = {\tt JA}[j+n]$;
\item[(c)] if ${\tt jaef}(i,j,n+1)$ holds then
${\tt jaef} (i,j,n)$ holds.
\end{itemize}
Verifying all three of these
then constitutes a proof by induction on $n$ that
{\tt jaef} is correct.
\begin{verbatim}
eval jaef_correct1 "?msd_nara Ai,j $jaef(i,j,0)":
eval jaef_correct2 "?msd_nara Ai,j,n $jaef(i,j,n) =>
   ($jaef(i,j,n+1) <=> JA[i+n]=JA[j+n])":
eval jaef_correct3 "?msd_nara Ai,j,n $jaef(i,j,n+1) =>
   $jaef(i,j,n)":
\end{verbatim}

Define $f(n)$ to be the smallest period of $x$ over all factors $x$ of
length $n$.
Recall that the {\it asymptotic critical exponent} of a sequence
is defined to be
$$ \limsup_{n \geq 1} {n \over {f(n)}}.$$

\begin{theorem}
The critical exponent of $\bf ja$ is $4$, and is attained only at
the two factors $0000$ and $1111$.

The asymptotic critical exponent of $\bf ja$ is $2$.  In fact,
all factors of length $n$ and period $p$ satisfy the inequality
$n \leq 2p+2$, and there exist arbitrarily large factors
where $n = 2p+2$.
\end{theorem}

\begin{proof}
We use the following {\tt Walnut} code.
\begin{verbatim}
def japer "?msd_nara p>0 & p<=n & $jaef(i,i+p,n-p)":
# 4457 states
eval jack0 "?msd_nara Ai,p,n (p>0 & n>=4*p & $japer(i,n,p)) => p=1":
# returns TRUE
eval jack1 "?msd_nara ~Ei,n,p $japer(i,n,p) & n>2*p+2":
# returns TRUE
eval jack2 "?msd_nara ~Ei,n,p $japer(i,n,p) & n>2*p+1":
# returns FALSE
eval jack3 "?msd_nara Am Ei,n,p (n>m) & $japer(i,n,p) & n=2*p+2":
# arbitrarily large factors of length n and period p, n=2p+2
\end{verbatim}
\end{proof}

We now turn to the subword complexity of $\tt ja$; it is
rather complicated.

\begin{theorem}
There are two $31$-state automata {\tt fdiff10} and {\tt fdiff12}
that accept exactly those
$n$ for which $\rho_{\bf ja} (n+1) - \rho_{\bf ja} (n) = 10$ (resp., $12$).
Every $n\geq 4$ has one of these two possibilities.
\label{scja}
\end{theorem}

\begin{proof}
We guess the automata {\tt fdiff10} and {\tt fdiff12} 
from empirical data and then
verify they are correct, as follows:

The first difference $\rho_{\bf ja} (n+1) - \rho_{\bf ja} (n)$
is easily seen to be the number of right-special factors of length
$n$ of $\bf ja$, so we can count these using a first-order
predicate for right-special factors similar to the one we used for $\bf n$.
A linear representation counting them is computed by the following
{\tt Walnut} code.  We then subtract the linear representation
so obtained from the one specified by the formula. When we minimize
the resulting linear representation, it minimizes to $0$, so the
two functions are the same.

\begin{verbatim}
def novel_ja "?msd_nara Aj (j<i) => ~$jaef(i,j,n)":
# 217 states
def jars n "?msd_nara $novel_ja(i,n) & Ej $jaef(i,j,n) & JA[i+n]!=JA[j+n]":
# 128 states
# right-special
def guess n "?msd_nara (n=0&i<1)|(n=1&i<2)|(n=2&i<4)|(n=3&i<8)|
   (n>=4&$fdiff10(n) &i<10)|(n>=4&$fdiff12(n)&i<12)":
# 95 states
\end{verbatim}
\end{proof}

We can also do a little additive number theory with {\bf ja}.
\begin{theorem}
Let $J_i = \{ n \geq 0 \suchthat {\bf ja}[n] = i \}$
for $i \in \{0,1 \}$.
Then every integer $\geq 10$ is the sum of two integers
in $J_0$, and every integer $\geq 2$ is the sum of two
integers in $J_1$.
\end{theorem}

\begin{proof}
We use the following code:
\begin{verbatim}
eval j0_sum "?msd_nara An (n>=10) => Ei,j n=i+j & JA[i]=@0 & JA[j]=@0":
eval j1_sum "?msd_nara An (n>=2) => Ei,j n=i+j & JA[i]=@1 & JA[j]=@1":
\end{verbatim}
\end{proof}


\section{Final words}
\label{finalw}

As mentioned in the introduction, the morphism $\nu$ and
the word $\bf n$ have a lot in
common with other, more well-studied morphisms and
infinite words.  Indeed, the word $\bf n$
can be viewed as the case $k = 3$ of the following locally
catenative definition of a family of sequences:  set $X_{-i}^k = 0$
for $0 \leq i < k$, and define
$X_i^{k} = X_{i-1}^{k} \overline{X_{i-k}^{k}}$ for $i \geq 1$.
It is easy to see that the limiting word ${\bf x}_k$ exists,
and is generated as the image
under $0\rightarrow 0$, $i \rightarrow 1$ for $1 \leq i < k$,
$0' \rightarrow 1$, $i' \rightarrow 0$ for $1 \leq i < k$
of the fixed point, starting with $0$,
of the morphism $\xi_k$
\begin{align*}
0 & \rightarrow 01\\
1 & \rightarrow 2\\
2 & \rightarrow 3\\
  & \vdots \\
(k-1) & \rightarrow 0' \\
0' &  \rightarrow 0' 1' \\
1' & \rightarrow 2'\\
2' & \rightarrow 3'\\
  & \vdots \\
(k-1)' & \rightarrow 0 \\
\end{align*}
which also gives the definition of $2k$-state automaton in the numeration
system based on the lengths of $|\xi^i(0)|$ for $i \geq 0$.
Alternatively, the $n$'th bit of  ${\bf x}_k $
is the parity of the number of $1$'s in the representation of $n$
in this numeration system.

The Thue-Morse morphism $\mu$ and 
word $\bf t$ correspond to the case $k=1$. 
The Fibonacci-Thue-Morphism morphism and word \cite{Ferrand:2007,Shallit:2021b,Shallit:2023b} correspond
to the case $k = 2$.

It is well-known that the Thue-Morse morphism is overlap-free, which
is to say that it contains no factors of length $2n+1$ and period $n$.
It is perhaps less well-known, but easy to prove with
{\tt Walnut}, that the Fibonacci-Thue-Morse word contains no factors
of length $2n+2$ and period $n$.  This is easy to prove as follows:
\begin{verbatim}
reg odd1 {0,1} "0*(10*10*)*10*":
def ftm "?msd_fib $odd1(n) & n>=0":
combine FTM ftm=1:
def ftmef "?msd_fib Au,v (u>=i & u<i+m & u+j=v+i) => FTM[u]=FTM[v]":
def ftm_isaper "?msd_fib p>0 & p<=m & $ftmef(i,i+p,m-p)":
eval ftmtest "?msd_fib Ei,n n>=1 & $ftm_isaper(i,2*n+2,n)":
\end{verbatim}
Here {\tt ftmtest} alleges the existence of a factor of length $2n+2$
and period $n$, and {\tt Walnut} returns {\tt FALSE}.  So there are
no such factors.

In this paper we have shown that the Allouche-Johnson word contains
no factors of length $2n+3$ and period $n$.  This suggests the following
conjecture:
\begin{conjecture}
Let $k \geq 1$. 
The infinite word ${\bf x}_k$ has critical exponent $k+1$, which is
attained by the words $0^{k+1}$ and $1^{k+1}$.  It
contains no factor of length $2n+k$
and period $n$, and therefore has asymptotic critical
exponent $2$.
\end{conjecture}

Another conjecture concerns the first difference 
$d_k(n) = \rho_{{\bf x}_k}(n+1) - \rho_{{\bf x}_k}(n)$
of the subword complexity function for ${\bf x}_k$.  For Thue-Morse
it is well-known
\cite{Brlek:1989,deLuca&Varricchio:1989,Avgustinovich:1994}
that $d_1 (n) \in \{ 2,4 \}$ for $n \geq 1$.
For the Fibonacci-Thue-Morse word, we know that
$d_2(n) \in \{ 6,8 \}$ for $n \geq 5$ from \cite{Shallit:2021b}.
For the Allouche-Johnson word, we proved in Theorem~\ref{scja} that
$d_3 (n) \in \{ 10, 12 \}$ for $n \geq 4$. This suggests
the following conjecture:
\begin{conjecture}
The first difference of the subword complexity function,
$\rho_{{\bf x}_k}(n+1)-\rho_{{\bf x}_k}(n)$ 
for ${\bf x}_k$, for $n$ large enough, takes
the values $4k-2$ and $4k$ only.
\end{conjecture}

For more about {\tt Walnut}, see the website\\
\centerline{\url{https://cs.uwaterloo.ca/~shallit/walnut.html} .}

\section*{Acknowledgments}
I thank Pierre Letouzey, Edita Pelantov\'a, Ingrid Vukusic, Robbert Fokkink,
and Jean-Paul Allouche for helpful remarks.


\begin{thebibliography}{99}

\bibitem{Allouche&Johnson:1996}
J.-P. Allouche and T. Johnson.
\newblock Narayana's cows and delayed morphisms.
\newblock In G. Assayag, M. Chemillier, and C. Eloy,
{\it Troisi\`emes Journ\'ees d'Informatique Musicale}, JIM '96.
\newblock \^Ile de Tatihou, France, 1996, pp. 2--7.

\bibitem{Ambroz&Masakova&Pelantova&Frougny:2006}
P. Ambro\v{z}, Z. Mas\'akov\'a, E. Pelantov\'a, and C. Frougny.
\newblock Palindromic complexity of infinite words
associated with simple Parry numbers.
\newblock {\it Ann. Inst. Fourier, Grenoble} {\bf 56} (2006),
2131--2160.

\bibitem{Avgustinovich:1994}
S.~V. Avgustinovich.
\newblock The number of different subwords of given length in the
  {M}orse-{H}edlund sequence.
\newblock {\em Sibirsk. Zh. Issled. Oper.} {\bf 1} (2) (1994), 3--7, 103.

\bibitem{Ballot:2017}
C. Ballot.
\newblock On functions expressible as words on a pair of Beatty sequences.
\newblock {\it J. Integer Sequences} {\bf 20} (2017), Article 17.4.2.

\bibitem{Berend&Frougny:1994}
D. Berend and C. Frougny.
\newblock Computability by finite automata and Pisot bases.
\newblock {\it Math. Systems Theory} {\bf 27} (1994), 275--282.

\bibitem{Bernat&Masakova&Pelantova:2007}
J. Bernat, Z. Mas\'akov\'a, and E. Pelantov\'a.
\newblock On a class of infinite words with affine factor complexity.
\newblock {\it Theoret. Comput. Sci.} {\bf 389} (2007), 12--25.

\bibitem{Berstel:1986b}
J.~Berstel.
\newblock Fibonacci words---a survey.
\newblock In G.~Rozenberg and A.~Salomaa, editors, {\em The Book of L}, pp.
  13--27. Springer-Verlag, 1986.

\bibitem{Berstel&Reutenauer:2011}
J.~Berstel and C.~Reutenauer.
\newblock {\em Noncommutative Rational Series with Applications}, Vol.~137 of
{\em Encyclopedia of Mathematics and Its Applications}.
\newblock Cambridge University Press, 2011.

\bibitem{Brlek:1989}
S. Brlek.
\newblock Enumeration of factors in the {Thue}-{Morse} word.
\newblock {\em Discrete Appl. Math.} {\bf 24} (1989), 83--96.

\bibitem{Crilly:2024}
T. Crilly.
\newblock Narayana's integer sequence revisited.
\newblock {\it Math. Gazette} {\bf 108} (2024), 262--269.

\bibitem{Crochemore:1981}
M. Crochemore.
\newblock An optimal algorithm for computing the repetitions in a word.
\newblock {\it Info. Proc. Letters} {\bf 12} (1981), 244--250.

\bibitem{Dekking&Shallit&Sloane:2020}
F. M. Dekking, J. Shallit, and N. J. A. Sloane.
\newblock Queens in exile:  non-attacking queens on infinite chess boards.
\newblock {\it Elect. J. Combinatorics} {\bf 27} (1) (2020), \#P1.52.

\bibitem{Dilcher:1993}
K. Dilcher.
\newblock On a class of iterative recurrence relations.
\newblock In G. E. Bergum, A. N. Philippou, and A. F. Horadam,
eds., {\it Applications of Fibonacci Numbers}, Vol.~5,
Springer, 1993, pp.~143--158.

\bibitem{Ericksen&Anderson:2012}
L. Ericksen and P. G. Anderson.
\newblock Patterns in differences between rows in $k$-Zeckendorf arrays.
\newblock {\it Fibonacci Quart.} {\bf 50} (2012), 11--18.

\bibitem{Ferrand:2007}
E. Ferrand.
\newblock An analogue of the Thue-Morse sequence.
\newblock {\it Elect. J. Combinatorics} {\bf 14} (2007), Paper \#R30.

\bibitem{Fraenkel:1985}
A. S. Fraenkel.
\newblock Systems of numeration.
\newblock {\it Amer. Math. Monthly} {\bf 92} (1985), 105--114.

\bibitem{Frougny:1992}
C. Frougny.
\newblock Representations of numbers and finite automata.
\newblock {\it Math. Systems Theory} {\bf 25} (1992), 37--60.

\bibitem{Frougny&Masakova&Pelantova:2004}
C. Frougny, Z. Mas\'akov\'a, and E. Pelantov\'a.
\newblock Complexity of infinite words associated with
beta-expansions.
\newblock {\it RAIRO-Inf. Theor. Appl.} {\bf 38} (2004), 163--185.
Corrigendum, {\bf 38} (2004), 269--271.

\bibitem{Frougny&Solomyak:1992}
C. Frougny and B. Solomyak.
\newblock Finite beta-expansions.
\newblock {\it Ergodic Theory Dynam. Systems} {\bf 12} (1992),
713--723.

\bibitem{Fujita&Machida:1986}
M. Fujita and K. Machida.
\newblock Electrons on one-dimensional quasi-lattices.
\newblock {\it Solid State Commun.} {\bf 59} (2) (1986), 61--65.

\bibitem{Hofstadter:1979}
D. Hofstadter.
\newblock {\it G\"odel, Escher, Bach:  An Eternal Golden Braid.}
\newblock Basic Books, 1979.

\bibitem{Hoggatt&Bicknell-Johnson:1982}
V. E. Hoggatt, Jr. and M. Bicknell-Johnson.
\newblock Lexicographic ordering and Fibonacci representations.
\newblock {\it Fibonacci Quart.} {\bf 20} (1982), 193--218.

\bibitem{Hopcroft&Ullman:1979}
J.~E. Hopcroft and J.~D. Ullman.
\newblock {\em Introduction to Automata Theory, Languages, and Computation}.
\newblock Addison-Wesley, 1979.

\bibitem{Karhumaki:1983}
J.~{Karhum\"aki}.
\newblock On cube-free $\omega$-words generated by binary morphisms.
\newblock {\em Disc. Appl. Math.} {\bf 5} (1983), 279--297.

\bibitem{Kimberling:1995b}
C. Kimberling.
\newblock The Zeckendorf array equals the Wythoff array.
\newblock {\it Fibonacci Quart.} {\bf 33} (1995), 3--8.

\bibitem{Kimberling&Moses:2010}
Clark Kimberling and Peter J. C. Moses.
\newblock Complementary equations and Zeckendorf arrays.
\newblock In {\it Applications of Fibonacci Numbers}, Vol.~10.
\newblock Proceedings of the Thirteenth International Conference on Fibonacci Numbers and Their Applications, William Webb, editor.
\newblock {\it Congressus Numerantium} {\bf 201} (2010) 161--178.

\bibitem{Knuth&Morris&Pratt:1977}
D.~E. Knuth, J.~Morris, and V.~Pratt.
\newblock Fast pattern matching in strings.
\newblock {\em SIAM J. Comput.} {\bf 6} (1977), 323--350.

\bibitem{Lekkerkerker:1952}
C.~G. Lekkerkerker.
\newblock Voorstelling van natuurlijke getallen door een som van getallen van
  {Fibonacci}.
\newblock {\em Simon Stevin} {\bf 29} (1952), 190--195.

\bibitem{Letouzey:2025}
P. Letouzey.
\newblock Generalized Hofstadter functions $G$, $H$ and
beyond:  numeration systems and discrepancy.
\newblock Arxiv preprint arXiv:2502.12615 [cs.DM],
February 19 2025.  Available at
\url{https://arxiv.org/abs/2502.12615}.

\bibitem{Letouzey:2024}
P. Letouzey, S. Li, and W. Steiner.
\newblock Pointwise order of generalized 
Hofstadter functions $G$, $H$ and beyond.
\newblock Arxiv preprint arXiv:2410.00529 [cs.DM],
October 1 2024. Available at
\url{https://arxiv.org/abs/2410.00529}.

\bibitem{Lin:2021}
X. Lin.
\newblock On the recurrence properties of Narayana's cows sequence.
\newblock {\it Symmetry} {\bf 13} (2021), Paper 149.

\bibitem{deLuca&Varricchio:1989}
A. de~Luca and S. Varricchio.
\newblock Some combinatorial properties of the {T}hue-{M}orse sequence and a
  problem in semigroups.
\newblock {\em Theoret. Comput. Sci.} {\bf 63} (1989), 333--348.

\bibitem{Meek&VanRees:1984}
D. S. Meek and G. H. J. Van Rees.
\newblock The solution of an iterated recurrence.
\newblock {\it Fibonacci Quart.} {\bf 22} (1984), 101--104.

\bibitem{Mousavi:2016}
H.~Mousavi.
\newblock Automatic theorem proving in {Walnut}.
\newblock Preprint, available at \url{http://arxiv.org/abs/1603.06017}, 2016.

\bibitem{Phunphayap}
P. N. Phunphayap, P. Pongsriiam, and J. Shallit.
\newblock Sumsets associated with Beatty sequences.
\newblock {\it Discrete Math.} {\bf 345} (2022), 112810.

\bibitem{Recht&Rosenman:1947}
L. Recht and M. Rosenman.
\newblock Problem 4247:  a difference sequence.
\newblock {\it Amer. Math. Monthly} {\bf 54} (1947), 232.  Solution in
{\bf 55} (1948), 588--592.

\bibitem{Shallit:2021}
J. Shallit.
\newblock Synchronized sequences.
\newblock In T. Lecroq and S. Puzynina, eds., {\it WORDS 2021},
Lecture Notes in Comp. Sci., Vol.~12847, Springer, 2021, pp.~1--19.

\bibitem{Shallit:2021b}
J. Shallit.
\newblock Subword complexity of the Fibonacci-Thue-Morse sequence: the proof of Dekking's conjecture.
\newblock {\it Indag. Math.} {\bf 32} (2021), 729--735.

\bibitem{Shallit:2022}
J. Shallit.
\newblock Sumsets of Wythoff sequences, Fibonacci representation, and beyond.
\newblock {\it Period. Math. Hung.} {\bf 84} (2022), 37--46.

\bibitem{Shallit:2023}
J.~Shallit.
\newblock {\em The Logical Approach To Automatic Sequences: Exploring
  Combinatorics on Words with {\tt Walnut}}, Vol. 482 of {\em London Math. Soc.
  Lecture Note Series}.
\newblock Cambridge University Press, 2023.

\bibitem{Shallit:2023b}
J. Shallit.
\newblock Note on a Fibonacci parity sequence.
\newblock {\it Cryptography and Commun.} {\bf 15} (2023), 309--315.

\bibitem{Sirvent&Wang:2002}
V. F. Sirvent and Y. Wang.
\newblock Self-affine tiling via substitution dynamical systems and
Rauzy fractals.
\newblock {\it Pacific J. Math.} {\bf 206} (2002), 465--485.


\bibitem{Sloane:2025}
N. J. A. Sloane et al.
\newblock The On-Line Encyclopedia of Integer Sequences.
\newblock Available at \url{https://oeis.org}, 2025.

\bibitem{Zeckendorf:1972}
E.~Zeckendorf.
\newblock {Repr\'esentation} des nombres naturels par une somme de nombres de
  {Fibonacci} ou de nombres de {Lucas}.
\newblock {\em Bull. Soc. Roy. {Li\`ege}} {\bf 41} (1972), 179--182.

\end{thebibliography}
\end{document}